\documentclass[11pt]{article}
\pdfoutput=1
\usepackage{amsthm,amsfonts,amssymb,amsmath}
\usepackage{caption,graphicx}
\usepackage{paralist}
\usepackage[text={6.75in,9.5in},centering,letterpaper]{geometry}
\usepackage[numbers,comma,square,sort&compress]{natbib}

\setcaptionmargin{0.25in}

\newtheorem{theorem}{Theorem}[section]
\newtheorem{proposition}[theorem]{Proposition}
\newtheorem{corollary}[theorem]{Corollary}
\newtheorem{lemma}[theorem]{Lemma}
\newtheorem{hypothesis}[theorem]{Hypothesis}
\newtheorem{definition}[theorem]{Definition}
\newtheorem{remark}[theorem]{Remark}
\numberwithin{equation}{section}

\def\Re{\mathop\mathrm{Re}\nolimits}		
\def\Im{\mathop\mathrm{Im}\nolimits}		
\def\Rg{\mathop\mathrm{Rg}\nolimits}		
\def\Ns{\mathop\mathrm{N}\nolimits}			
\def\graph{\mathop\mathrm{graph}\nolimits}		
\def\errfn{\mathop\mathrm{errfn}\nolimits}		

\newcommand{\rmO}{\mathrm{O}}		
\newcommand{\rmd}{\mathrm{d}}		
\newcommand{\rme}{\mathrm{e}}		
\newcommand{\rmi}{\mathrm{i}}		

\begin{document}

\title{Nonlinear stability of source defects in oscillatory media}
\author{Margaret Beck\footnotemark[1] \and Toan T. Nguyen\footnotemark[2] \and Bj\"orn Sandstede\footnotemark[3] \and Kevin Zumbrun\footnotemark[4]}

\maketitle

\renewcommand{\thefootnote}{\fnsymbol{footnote}}

\footnotetext[1]{Department of Mathematics and Statistics, 
Boston University, Boston, MA, 02215. Email: mabeck@bu.edu}

\footnotetext[2]{Department of Mathematics, Penn State University, State College, PA 16803. Email: nguyen@math.psu.edu}

\footnotetext[3]{Division of Applied Mathematics, Brown University, Providence, RI 02912. 
 Email: bjorn$\underline{~}$sandstede@brown.edu}

\footnotetext[4]{Department of Mathematics, Indiana University,  
Bloomington, IN 47405. Email: kzumbrun@indiana.edu}

\begin{abstract}
In this paper, we prove the nonlinear stability under localized perturbations of spectrally stable time-periodic source defects of reaction-diffusion systems. Consisting of a core that emits periodic wave trains to each side, source defects are important as organizing centers of more complicated flows. Our analysis uses spatial dynamics combined with an instantaneous phase-tracking technique to obtain detailed pointwise estimates describing perturbations to lowest order as a phase-shift radiating outward at a linear rate plus a pair of localized approximately Gaussian excitations along the phase-shift 
boundaries;  we show that in the wake of these outgoing waves the perturbed solution converges time-exponentially to a space-time translate of the original source pattern.
\end{abstract}

\begin{center}
\end{center}



\newpage 

\section{Introduction}

We are interested in the stability properties of interfaces between stable spatially periodic structures with possibly different wavenumbers: we refer to such interfaces as defects and to the asymptotic spatially periodic travelling waves as wave trains; see Figure~\ref{f:1} for an illustration. Typically, wave trains and defects will depend on time, and we are particularly interested in structures that are periodic in time, possibly after transforming into a comoving reference frame. Defect solutions arise in many biological, chemical, and physical processes: examples are planar spiral waves \cite{swinney,nettesheim}, flip-flops in chemical reactions \cite{perraud}, and surface waves in hydrothermal fluid flows \cite{pastur}.

\begin{figure}[b]
\centering\includegraphics[scale=1]{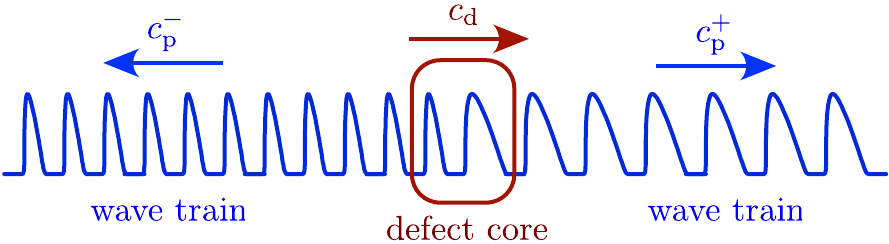} 
\caption{A defect moving with speed $c_\mathrm{d}$ that connects two spatially periodic travelling waves, referred to as wave trains, with phase speeds $c_\mathrm{p}^\pm$ at $x=\pm\infty$, respectively.}
\label{f:1}
\end{figure}

\subsection*{Wave trains and defects}

Before we discuss the specific goals of this paper in more detail, we make the notion of defects and wave trains more precise. We consider reaction-diffusion systems of the form
\begin{equation}\label{eqn:rd}
u_t = Du_{xx} + f(u),
\quad
\hbox{\rm $(x,t)\in\mathbb{R}\times\mathbb{R}^+$, \quad $u\in \mathbb{R}^n$},
\end{equation}
where $D \in \mathbb{R}^{n\times n}$ is an invertible, diagonal diffusion matrix. Systems of this form often exhibit one-parameter families of spatially-periodic travelling waves that are parametrized by their spatial wavenumber $k$: thus, we assume that there are wave-train solutions of the form
\[
u(x,t) = u_\mathrm{wt}(kx-\omega_\mathrm{nl}(k) t;k)
\]
for $k$ in an open, nonempty interval, where the profile $u_\mathrm{wt}(\theta;k)$ is $2\pi$-periodic in $\theta$. Here, $\omega=\omega_\mathrm{nl}(k)$ denotes the temporal frequency, which will be a function of the wavenumber $k$, and it is typically referred to as the nonlinear dispersion relation. Wave trains therefore propagate with the phase velocity $c_\mathrm{p}=\omega_\mathrm{nl}(k)/k$. An interesting quantity associated with a wave train is its group velocity $c_\mathrm{g}$, which is defined by the derivative of the nonlinear dispersion relation:
\[
c_\mathrm{g} = \frac{\rmd\omega_\mathrm{nl}(k)}{\rmd k}.
\]
The group velocity turns out to be the speed with which small localized perturbations of a wave train propagate along the wave train as a function of time $t$; see \cite{DSSS,SSSU,JNRZ2} for further discussion and proofs.

We now turn to the definition of defect solutions. Following \cite{Sa,SH,Hecke,SS04a}, a defect is a solution of (\ref{eqn:rd}) of the form
\[
u(x,t) = \bar{u}(x-c_\mathrm{d}t,t),
\]
where the defect profile $\bar{u}(\xi,t)$ is assumed to be periodic in $t$ and, for appropriate wavenumbers $k_\pm$, we have that
\[
\bar{u}(\xi,t) \to u_\mathrm{wt}\left(k_\pm\xi-(\omega_\mathrm{nl}(k_\pm)-c_\mathrm{d}k_\pm)t;k_\pm\right) \mbox{ as } \xi\to\pm\infty
\]
uniformly in $t$ so that the defect converges to (possibly different) wave trains in the far field, that is, as $\xi\to\pm\infty$. We remark that periodicity of the defect profile in time implies that $\omega_\mathrm{nl}(k_\pm)-c_\mathrm{d}k_\pm = \omega_\mathrm{d}$, with $\frac{2\pi}{\omega_\mathrm{d}}$ being the time periodicity of the defect, and hence the defect velocity $c_\mathrm{d}$ is determined by the Rankine--Hugoniot condition
\[
c_\mathrm{d} = \frac{\omega_\mathrm{nl}(k_+)-\omega_\mathrm{nl}(k_-)}{k_+-k_-} . 
\]
The properties of these coherent structures have been analyzed in detail in \cite{SS04a}, matching experimental observations \cite{perraud, Yon}. In particular, depending on the group velocities $c_\mathrm{g}(k_\pm)$ of the asymptotic wave trains, defects can be classified into distinct types \cite{SH,Hecke,SS04a} that have different multiplicity, robustness, and stability properties. Following \cite{Sa,SH,Hecke,SS04a}, we distinguish

\qquad
\begin{tabular}{llcl}
sinks: & exist for arbitrary $k_\pm$ & when &                  $c_\mathrm{g}(k_-)>c_\mathrm{d}>c_\mathrm{g}(k_+)$ \\
transmission defects: & exist for arbitrary $k_+=k_-$ & when & $c_\mathrm{g}(k_\pm)>c_\mathrm{d}$ or
                                                               $c_\mathrm{g}(k_\pm)<c_\mathrm{d}$ \\
contact defects: & exist for arbitrary $k_+=k_-$ & when &      $c_\mathrm{g}(k_\pm)=c_\mathrm{d}$ \\
sources: & exist for unique $k_\pm$ & when &                   $c_\mathrm{g}(k_-)<c_\mathrm{d}<c_\mathrm{g}(k_+)$.
\end{tabular}

As discussed in \cite{SS04a,Sa,BNSZ2}, sinks can be thought of as passive interfaces that accommodate two colliding wave trains. Similarly, transmission and contact defects accommodate phase-shift dislocations within wave-train solutions. Source defects, on the other hand, occur only for discrete wavenumbers $k_\pm$ and therefore select the wavenumbers of the wave trains that emerge from the defect core into the surrounding medium: hence, they may be thought of as organizing the surrounding global dynamics, rather than the reverse. In their comprehensive review on general pattern formation phenomena, Cross and Hohenberg \cite[pp. 855-857]{CH} emphasize the importance of sources as pattern selection mechanisms. In this sense, source defects are of particular interest from a dynamical point of view.

Yet, among the main non-characteristic varieties (sinks, transmission defects, and sources), source defects are the only type whose stability properties are not understood mathematically. Generic aspects of spectral stability of all types have been examined in \cite{SS04a,SS04b}. Moreover, it has been shown that spectral stability implies nonlinear stability for sinks in \cite{SS04a} and transmission defects in \cite{GSU}, settling the question of linear and nonlinear stability in these cases. However, both of these analyses utilize weighted-norm techniques that do not seem to apply to source defects. Thus, establishing nonlinear stability in this phenomenologically important case is an important open problem, and resolving this question is the purpose of the present work.

We note that, in our previous recent works \cite{BNSZ1,BNSZ2}, we considered two simpler scenarios that allowed us to gain some insight into the difficulties that we anticipated to see in the general case of source defects in reaction-diffusion systems: these simpler cases consisted of a modified Burgers equations and the cubic-quintic complex Ginzburg--Landau equation. We will discuss these cases and the differences from the general case studied here in more detail below.

\subsection*{Stability of sources}

We begin with a heuristic discussion of the anticipated stability properties of sources. As outlined above, sources are characterized by the feature that the group velocities of the asymptotic wave trains point away from the core of the defect. Thus, localized perturbations added in the far field will not affect the defect as they will propagate away from the defect core and decay to zero if the asymptotic wave trains are stable. If the defect core is subjected to a localized perturbation, we expect that it may change its position and adjust the phase of the wave trains it emits. The result is that the defect core will emit wave trains with a different phase from a new position. Thus, we expect that a localized perturbation of the core will lead to a phase front that propagates with the group velocity of the asymptotic wave trains into the far field to accommodate the difference between the phases of the wave trains before and after the localized perturbation was added.

We now state our hypotheses and results more precisely and refer to \S\ref{s:prelim} for further details. We focus on a given source and first transform the reaction-diffusion system (\ref{eqn:rd}) into a spatial coordinate system that moves with the velocity $c_\mathrm{d}$ of the source to get
\begin{equation}\label{e:1}
u_t = Du_{xx} + c_\mathrm{d} u_x + f(u),
\end{equation}
where we use the same variable $x$ to denote the new comoving spatial coordinate. In the comoving frame, the source defect is then given by $u(x,t)=\bar{u}(x,t)$, where the profile $\bar{u}(x,t)$ is periodic in $t$, and the asymptotic wave trains are given by
\[
u(x,t) = u_\mathrm{wt}(k_\pm x-(\omega_\mathrm{nl}(k_\pm)-c_\mathrm{d}k_\pm)t;k_\pm)
\]
with group velocities
\[
c_\pm := c_\mathrm{g}(k_\pm)-c_\mathrm{d}.
\]
Next, we assume that the wave trains are spectrally stable: more precisely, we assume that the Floquet-Bloch spectrum of the linearization
\[
v_t = Dv_{xx} + c_\mathrm{d} v_x + f_u(u_\mathrm{wt}(k_\pm x-(\omega_\mathrm{nl}(k_\pm)-c_\mathrm{d}k_\pm)t;k_\pm)) v
\]
of (\ref{e:1}) about the wave trains in the closed left half-plane consists of the curve
\begin{equation}\label{spectrum}
\lambda_\pm(\gamma) = - \rmi c_\pm \gamma - \mathrm{d}_\pm \gamma^2 + \rmO(|\gamma|^3) \qquad \mbox{for } \gamma\approx0,
\end{equation}
where the coefficients $\mathrm{d}_\pm>0$ are assumed to be positive so that $\Re\lambda_\pm(\gamma)\leq0$ for all $\gamma$ close to zero, as well as other curves that lie to the left of the line $\Re\lambda=-\delta$ for some $\delta>0$. Finally, we assume spectral stability of the source defect: setting
\[
L^2_a(\mathbb{R}) := \Big\{ u: \rme^{-a|x|} u(x) \in L^2(\mathbb{R}) \Big\},
\]
we assume that the Floquet-Bloch spectrum of the linearization
\begin{equation}\label{lin-intro}
v_t = Dv_{xx} + c_\mathrm{d}v_x+ f_u(\bar{u}(x, t))v
\end{equation}
about the source posed on the space $L^2_a(\mathbb{R})$ for a sufficiently small $a>0$ lies in the open left half-plane with the exception of two eigenvalues, counted with multiplicity, at the origin that correspond to the eigenfunctions $\bar{u}_x(x,t)$ and $\bar{u}_t(x,t)$.

\begin{figure}
\centering
\includegraphics[scale=0.5]{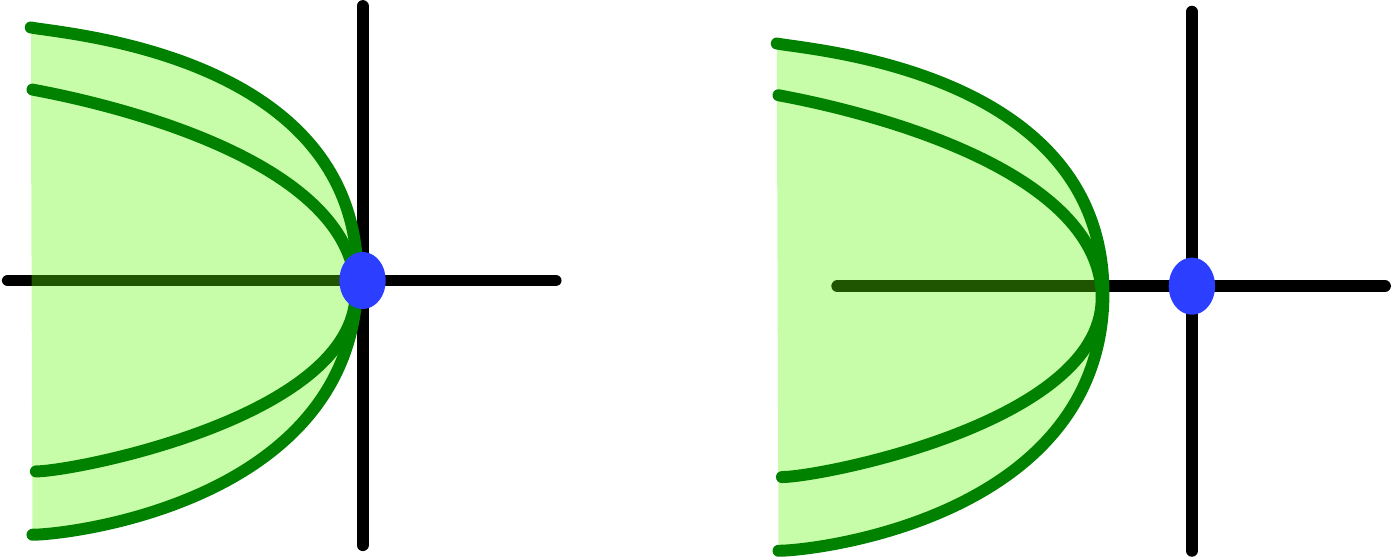}
\put(-140,70){$\mathbb{C}$}
\put(-20,70){$\mathbb{C}$}
\caption{Illustrated is the Floquet spectrum of the linearization (\ref{lin-intro}) on $L^2(\mathbb{R})$ (left) and $L^2_a(\mathbb{R})$ (right).}
\label{fig-spectrumdf}
\end{figure}

\subsection*{Main result}

Before we can state out main result, we need to introduce additional notation. Given the constants $c_\pm$ and $\mathrm{d}_\pm$ that we introduced above in \eqref{spectrum}, and a constant $M_0>0$ that will be determined later, we define the error-function plateau
\begin{equation}\label{def-epm}
e(x,t) := \errfn\left(\frac{x-c_+ t}{\sqrt{4\mathrm{\mathrm{d}_+}t}}\right)
	-  \errfn\left(\frac{x - c_- t}{\sqrt{4\mathrm{\mathrm{d}_-}t}}\right), \qquad
\errfn(z):= \frac{1}{2\pi} \int_{-\infty}^{z} \rme^{-x^2} dx,
\end{equation}
and the radiating Gaussian profile
\begin{equation}\label{Gaussian-like}
\theta(x,t) := \sum_\pm \frac{1}{(1+t)^{\frac12}} \rme^{-\frac{(x- c_\pm t)^2}{M_0(1+t)}}.
\end{equation}
The main result of this paper is then as follows.

\begin{theorem}\label{main}
If $f$ is of class $C^3$, the wave trains satisfy Hypothesis~\ref{h:wt}, and $\bar{u}(x,t)$ is a standing source defect of (\ref{e:1}) that satisfies Hypotheses~\ref{h:so} and~\ref{h:ss}, then there are constants $\varepsilon_0,b,C_0,C_1,M_0>0$ with the following properties: assume that $u(x,0)$ satisfies $\|\rme^{x^2/C_0} (u(x,0)-\bar{u}(x,0))\|_{C^2}=:\varepsilon\leq\varepsilon_0$, then the solution $u(x,t)$ of \eqref{e:1} with initial condition $u(x,0)$ exists globally in time, and there are functions $\varphi(x,t)$ and $\psi(x,t)$ such that
\begin{eqnarray}\label{e:profile_estimates}
& |u(x+\psi(x,t),t+\varphi(x,t)) - \bar{u}(x,t)| \leq \varepsilon C_1 \theta(x,t), &
\label{err_profile} \\ \label{err_phase}
& |\psi(x,t)| + |\varphi(x,t)| \leq \varepsilon C_1, \qquad
|\mathcal{D}_{x,t} \varphi(x,t)| + |\mathcal{D}_{x,t} \psi(x,t)| \leq \varepsilon C_1 \theta(x,t) &
\end{eqnarray}
for $t\geq0$, where $\mathcal{D}_{x,t}$ denotes $(\partial_x,\partial_t)$. Furthermore, there are smooth functions $\delta_\varphi(t),\delta_\psi(t)$ and constants $\delta_\varphi^\infty, \delta_\psi^\infty$ such that, for $t>0$,
\begin{eqnarray*}
& |\delta_\varphi^\infty| + |\delta_\psi^\infty| \leq \varepsilon C_1, \qquad
|\delta_\varphi(t) - \delta_\varphi^\infty| + |\delta_\psi(t) - \delta_\psi^\infty| \leq \varepsilon C_1 \rme^{-bt}, & \\
& |\varphi(x,t) - e(x,t) \delta_\varphi(t)| + |\psi(x,t) - e(x,t)\delta_\psi(t)| \leq \varepsilon C_1 (1+t)^{\frac12} \theta(x,t) .&
\end{eqnarray*}
\end{theorem}

\begin{figure}
\centering\includegraphics[scale=1]{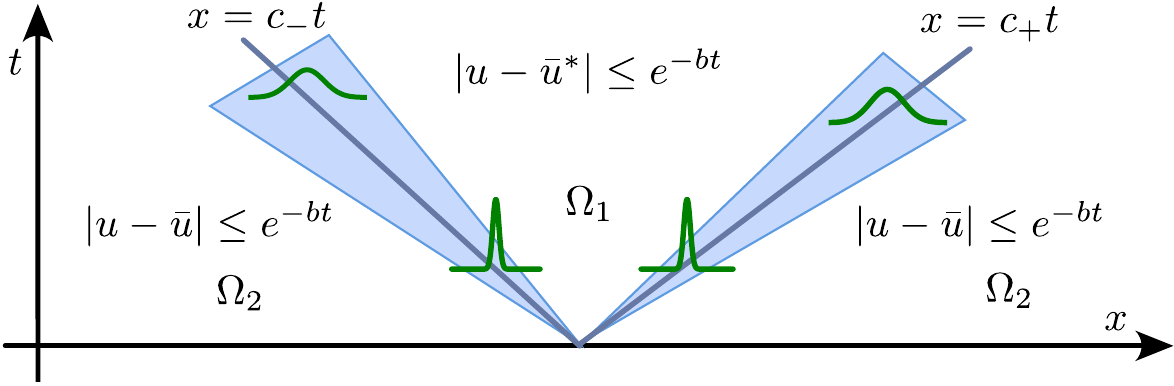} 
\caption{Shown is the strong point-wise convergence in space-time diagram in the comoving frame, where we define the shifted source $\bar{u}^*(x,t)$ by $\bar{u}(x-\delta_\psi^\infty,t-\delta_\varphi^\infty)$ with $\delta_\psi^\infty$ and $\delta_\varphi^\infty$ as in Theorem~\ref{main}. Along the rays $x\approx c_\pm t$, the perturbation will decay like a moving Gaussian.}
\label{fig-mainthm}
\end{figure}

The following corollary is a direct consequence of Theorem~\ref{main}; see Figure~\ref{fig-mainthm} for an illustration.

\begin{corollary}\label{cor-main}
Pick any constant $\varepsilon_1>0$. Under the assumptions of Theorem~\ref{main}, there are constants $b,C>0$ so that any solution $u(x,t)$ of (\ref{e:1}) that meets the assumptions of Theorem~\ref{main} satisfies
\begin{eqnarray*}
|u(x,t)-\bar{u}(x-\delta_\psi^\infty,t-\delta_\varphi^\infty)| & \leq & \varepsilon C \rme^{-bt}, \qquad (x,t)\in\Omega_1 \\
|u(x,t)-\bar{u}(x,t)| & \leq & \varepsilon C \rme^{-bt}, \qquad (x,t)\in\Omega_2
\end{eqnarray*}
where $\Omega_1=\{(x,y):(c_-+\varepsilon_1)t\leq x\leq (c_+-\varepsilon_1)t\}$ and 
$\Omega_2=\{(x,y):x\leq(c_--\varepsilon_1)t\mbox{ or } x\geq(c_++\varepsilon_1)t\}$.
\end{corollary}

The corollary confirms the heuristic stability picture outlined earlier in this section: away from the characteristic cones $x\approx c_\pm t$, solutions that arise as localized perturbations of a source defect converge exponentially in time to an appropriate space-time translate of the original source near the core and to the original source in the far field. Our results given above are not sharp enough to give a detailed description of the phase fronts that mediate between the phases of the original and the translated source defects inside the characteristic cones.

Before proceeding, we remark that the conclusions of Theorem~\ref{main} remain true for initial perturbations that decay exponentially in space instead of the stronger Gaussian decay we assumed. Our results can also be modified to allow for initial perturbations that decay only algebraically in space as in \cite{BSZ}, though we would then recover only $(1+t)^{-1}$ decay in the wake region instead of the time-exponential decay seen in Corollary~\ref{cor-main}. Finally, our arguments extend with further elaboration as in \cite{BSZ,RZ}, to quasilinear strictly parabolic diffusion $(D(u)u_x)_x$.


\subsection*{Comparison with other results, and outline of the proof of Theorem~\ref{main}}

In our previous works \cite{BNSZ1,BNSZ2}, we considered two simpler scenarios. As can be seen from the defect classification outlined above, sources have group velocities that are opposite to those of sinks. Sinks, in turn, can be thought of as the analogues of Lax shocks. Motivated by this analogy, we proved in \cite{BNSZ1} the nonlinear stability of the $\phi=0$ solution of the equation
\begin{equation}\label{e:toy}
\phi_t = \phi_{xx} - \tanh\left(\frac{cx}{2}\right) \phi_x + \phi_x^2,
\end{equation}
where we think of $\phi(x,t)$ as the phase of the asymptotic wave trains relative to the defect at $x=0$. Compared with the usual Burgers equation, the key difference is that the characteristic speeds in (\ref{e:toy}) point away from $x=0$, so that we can think of $\phi=0$ as a source defect, rather than a sink or Lax shock. In \cite{BNSZ2}, we proved the nonlinear stability of spectrally stable sources of the complex cubic-quintic Ginzburg--Landau equation
\begin{equation}\label{e:cgl}
A_t = (1+\rmi\alpha)A_{xx} + \mu A - (1+\rmi\beta) A|A|^2 +  (\gamma_1+\rmi\gamma_2) A|A|^4.
\end{equation}
In this case, sources are of the special form
\[
A(x,t) = r(x) \rme^{\rmi\varphi(x)} \rme^{-\rmi\omega_0t},
\]
so that their time dependence disappears in a corotating frame due to the gauge invariance $A\mapsto\rme^{\rmi\kappa}A$ of (\ref{e:cgl}). Our analysis utilized this property extensively to extract an explicit equation for the phase that  agreed to leading order with (\ref{e:toy}) and that we could therefore analyse in a similar fashion. The result stated in \cite[Theorem~1.1]{BNSZ2} agrees with Theorem~\ref{main}. However, in \cite{BNSZ2}, we were also able to resolve the shape of the phase fronts inside the characteristic cones: up to Gaussian error terms, the phase fronts are given by (\ref{def-epm}).

The proofs in \cite{BNSZ1,BNSZ2} rely on the fact that perturbations of a source defect satisfy a variation-of-constants formula that involves the spatio-temporal Green's function of the linearization about the source defect integrated against the nonlinear terms. The goal is then to show that solutions to the variation-of-constants formula exist that satisfy appropriate spatio-temporal estimates that show that perturbations behave as desired. To implement this idea, it is necessary to derive detailed pointwise bounds on the Green's function: we obtained these bounds by establishing expansions of the resolvent kernel of the linearization, which were then transferred to the Green's function using Laplace transforms. The strategy for tackling the general case of reaction-diffusion systems is similar, though substantial technical challenges arise that were not present in the simpler problems. Firstly, sources are genuinely time-dependent, and the connection between the time-dependent resolvent kernel and the Green's function of (\ref{lin-intro}) needs to be closely examined and resolved. Secondly, in the absence of gauge invariances, the phase is not explicitly determined by the structure of the equations but must instead be strategically defined in a way that captures the main asymptotic behavior of solutions.

We address the first issue by extending the approach taken in \cite{BSZ} in the case of time-periodic viscous Lax shocks with asymptotically constant end states to the case of source defects that converge to wave trains. The outcome of this analysis is the following expansion of the Green's function of the linearization (\ref{lin-intro}).

\begin{theorem}\label{linthm}
Under the assumptions of Theorem~\ref{main}, there are constants $C,\eta>0$ so that the followings hold. The Green's function $G(x,t;y,s)$ of the linearization \eqref{lin-intro}
\[
v_t = Dv_{xx} + c_\mathrm{d}v_x+ f_u(\bar{u}(x, t))v
\]
can be written as
\begin{equation}\label{Gdecomp}
 G(x,t;y,s) =  \bar{u}_x(x,t) E_1(x,t; y,s) + \bar{u}_t(x,t) E_2(x,t; y,s) + G_R(x,t; y,s)
\end{equation}
with
\[
E_j (x,t; y,s) = \chi(t)\big(  e(x-y,t-s) \beta_j(y) + G_j(x,t; y,s)\big), 
\]
where $e$ is as in \eqref{def-epm},
$\beta_j(y)$ is exponentially localized (that is, $|\beta_j(y)|\leq C\rme^{-\eta|y|}$), the remainder terms $G_j(x,t; y,s)$ 
are bounded by a moving Gaussian:
$$|G_j(x,t;y,s)|\leq C\theta(x-y,t-s),$$ 
and $\chi(t)$ is a smooth cut-off function that vanishes in $[0,1]$ and is equal to one for $t\geq2$. In addition, there hold the following derivative bounds, for $k=0,1$ and all $t\geq s$,
\begin{eqnarray*}
|\mathcal{D}_{x,t}^k G_R (x,t;y,s)| & \leq & C(t-s+1)^k(t-s)^{-k} ((t-s)^{-\frac12} + \rme^{-\eta |y|}) \theta(x-y,t-s)
\\
|\mathcal{D}_{x,t}^{1+k} G_j(x,t; y,s)| & \leq & C(t-s+1)^k(t-s)^{-k} ((t-s)^{-\frac12} + \rme^{-\eta |y|}) \theta(x-y,t-s).
\end{eqnarray*}
\end{theorem}

The main advance of the preceding theorem over the results we obtained in \cite{BNSZ1,BNSZ2} for the cases of (\ref{e:toy}) and (\ref{e:cgl}) is that the remainder term $G_R$ behaves like a differentiated Gaussian, and not only as a Gaussian as concluded in \cite{BNSZ1,BNSZ2}. As we will now discuss, this stronger result for the linearized equation is key for closing a nonlinear iteration scheme for the variation-of-constants formula.

If $u(x,t)$ is a solution that is initially close to the source defect, we introduce the spatial and temporal shifts $\psi(x,t)$ and $\varphi(x,t)$, respectively, and a profile adjustment $v(x,t)$ to compare $u(x,t)$ via
\[
u(x+\psi(x,t),t+\varphi(x,t)) = \bar{u}(x,t) + v(x,t)
\]
to the original source $\bar{u}(x,t)$. We will show that $(\psi,\varphi,v)$ satisfies the nonlinear system
\[
\left[\partial_t - D\partial^2_{x} - c_\mathrm{d} \partial_x - f_u(\bar{u}(x,t))\right] (v-\bar{u}_x\psi-\bar{u}_t\varphi) = \rmO((\mathcal{D}_{x,t}\phi,\mathcal{D}_{x,t}\psi,v)^2).
\]
Using the Green's function $G(x,t;y,s)$ of the linearization about the defect on the left-hand side, we can rewrite this equation as in variation-of-constants form as
\begin{eqnarray}
(v-\bar{u}_x\psi-\bar{u}_t\varphi)(x,t) & = & \int_\mathbb{R} G(x,t;y,0) (v-\bar{u}_x\psi-\bar{u}_t\varphi)(y,0)\, \rmd y
\label{e:voc} \\ \nonumber &&
+ \int_0^T \int_\mathbb{R} G(x,t;y,s) \rmO((\mathcal{D}_{x,t}\phi,\mathcal{D}_{x,t}\psi,v)^2)\, \rmd y\,\rmd s.
\end{eqnarray}
Inspecting Theorem~\ref{linthm}, we see that the term $G_R$ in the Green's function behaves like a differentiated Gaussian. If, as indicated in (\ref{err_profile}), the profile perturbation $v$ behaves like a Gaussian, then we integrate in the term $\int_0^T\int_\mathbb{R} \ldots \rmd y\, \rmd s$ an differentiated Gaussian $G_R$ against a squared Gaussian: the outcome is a function that behaves again like a Gaussian. Thus, if we ignore the terms in the Green's function that involve $\bar{u}_x$ and $\bar{u}_t$, the variation-of-constants formula maps Gaussians into Gaussians, and we can expect that the fixed-point also behaves like a Gaussian. In summary, the estimate for $G_R$ will allow us to show that profile perturbations indeed decay like Gaussians as claimed in (\ref{err_profile})---note that, if $G_R$ behaved only like a Gaussian, the double integral would produce a function that is not even bounded (this is related to the fact that solutions of $u_t=u_{xx}+u^2$ exhibit finite-time blow-up).

The decomposition (\ref{Gdecomp}) of the Green's function also allows us to extract equations for the spatial and temporal shifts $(\psi,\varphi)$ by separately combining all terms in (\ref{e:voc}) that multiply the functions $\bar{u}_x$ and $\bar{u}_t$, respectively, and requiring that the resulting two expressions vanish identically. This results in an equation of the form
\[
\psi(x,t) = -\int_\mathbb{R} E_1(x,t;y,0) (v-\bar{u}_x\psi-\bar{u}_t\varphi)(y,0)\, \rmd y - \int_0^T \int_\mathbb{R} E_1(x,t;y,s) \rmO((\mathcal{D}_{x,t}\phi,\mathcal{D}_{x,t}\psi,v)^2)\, \rmd y\,\rmd s
\]
for $\psi(x,t)$ and an analogous equation for $\varphi(x,t)$. Taking derivatives of these equations with respect to $x$ then allows us to use similar arguments to show that $\psi_x$ and $\varphi_x$ behave like Gaussians as stated in (\ref{err_phase}). We note that this approach to extracting phase equations is similar to the approach used in \cite{JZ} and elsewhere in the simpler spatially-periodic case.

\paragraph{Notation:} We shall use $C$ to denote a universal constant that may change from line to line but is independent of the initial data, space, and time. We also use the notation $f=\rmO(g)$ or $f\lesssim g$ to mean that $|f|\leq C|g|$. 


\section{Hypotheses}\label{s:prelim}

We now discuss our hypotheses in detail. Our goal is to give a streamlined version of the properties of these structures in the form needed in the remainder of this paper, and we refer to \cite{SS04a} for more background on wave trains and source defects. Throughout, we will work in the co-moving frame of a defect $\bar{u}(x-c_\mathrm{d}t,t)$, and we will also rescale space and time to set its time period to $2\pi$.

Thus, we consider a reaction-diffusion system
\begin{equation}\label{eqs-RD}
u_t = Du_{xx} + c_\mathrm{d} u_x + f(u)
\end{equation}
with $(x,t)\in\mathbb{R}\times\mathbb{R}^+$, $u\in\mathbb{R}^n$, and $D \in \mathbb{R}^{n\times n}$ an invertible, diagonal matrix. Our first hypothesis captures the assumption that $\bar{u}(x,t)$ is a defect solution that is $2\pi$-periodic in time and converges in space to two, possibly different, wave trains (spatially periodic traveling waves) as $x\to\pm\infty$.

\begin{hypothesis}\label{h:so}
Assume that $u(x,t)=u_\mathrm{wt}(k_\pm x-t;k_\pm)$ satisfies (\ref{eqs-RD}), where $k_\pm\neq0$ and $u_\mathrm{wt}(\theta;k_\pm)$ is $2\pi$-periodic in $\theta$. We also assume that $\bar{u}(x,t)$ is $2\pi$-periodic in $t$ and satisfies (\ref{eqs-RD}) as well as
\[
\bar{u}(x,t) - u_\mathrm{wt}(k_\pm x - t;k_\pm) \to 0, \qquad x\to \pm\infty
\]
uniformly in time.
\end{hypothesis}

Next, define the linear operators
\[
L_\pm := k_\pm^2 D\partial^2_{\theta} + (k_\pm c_\mathrm{d}+1) \partial_\theta + f_u(u_\mathrm{wt}(\theta;k_\pm))
\]
associated with the asymptotic wave trains. Following \cite{SS04a}, we make the following assumption:

\begin{hypothesis}\label{h:wt}
\begin{compactitem}
\item The operators $L_\pm$ posed on $L^2(\mathbb{T})$, with $\mathbb{T} = [0,2\pi]/\!\sim$, each have an algebraically simple eigenvalue at the origin $\lambda=0$.
\item It follows that the spectrum of $L_\pm$ posed on $L^2(\mathbb{R})$ near the origin consists of a smooth curve $\widetilde{\lambda}_\pm(\xi)=-\rmi (c_\pm+c_\mathrm{d}-1/k_\pm)\xi-\mathrm{d}_\pm\xi^2+\rmO(|\xi|^3)$ for appropriate real numbers $c_\pm$ and $\mathrm{d}_\pm$: we assume that $c_-<0<c_+$ and $\mathrm{d}_\pm>0$.
\item We assume that the spectrum $L_\pm$ posed on $L^2(\mathbb{R})$ is contained in the open left half-plane with the exception of the curve $\widetilde{\lambda}_\pm(\xi)$ for sufficiently small $\xi$.
\end{compactitem}
\end{hypothesis}

We now state a result that relates the spectra of the operators $L_\pm$ to the Floquet spectra of the linearization
\[
v_t = Dv_{xx} + c_\mathrm{d} v_x + f_u(u_\mathrm{wt}(k_\pm x-t;k_\pm)) v
\]
of the reaction-diffusion system (\ref{eqs-RD}) about the asymptotic wave trains $u_\mathrm{wt}(k_\pm x-t;k_\pm)$. We denote by $\Phi_\pm:v(x,0)\mapsto v(x,2\pi)$ the linear time-$2\pi$ maps posed on $L^2(\mathbb{R},\mathbb{R}^n)$ associated with the linear PDE. By definition, a complex number $\lambda$ is said to be in the Floquet spectrum of $\Phi_\pm$ if, and only if, $\rho=\rme^{2\pi\lambda}$ is in the spectrum of the operator $\Phi_\pm$. We then have the following result.

\begin{lemma}[\cite{SS04a}]\label{lem-ss}
Assume that Hypotheses~\ref{h:wt} is met, then the Floquet spectrum of the operator $\Phi_\pm$ on $L^2(\mathbb{R},\mathbb{R}^n)$ lies in the open left half-plane with the exception of the curve $\lambda_\pm(\xi)$ with expansion
\[
\lambda_\pm(\xi) = -\rmi c_\pm \xi - \mathrm{d}_\pm \xi^2 + \rmO(|\xi|^3)
\]
for $\xi\in\mathbb{R}$ close to zero.
\end{lemma}

In particular, as shown in \cite[\S3]{SS04a}, the numbers $c_\pm$ are the group velocities of the wave trains in the moving frame (\ref{eqs-RD}). The assumption on the signs of $c_\pm$ ensures that the solution $\bar{u}(x,t)$ of (\ref{eqs-RD}) is a source defect in the classification laid out in \cite{SS04a}. The next result shows that the convergence of the defect to the asymptotic wave trains is exponential.

\begin{lemma}[{\cite[Theorem~1.4 and Corollary~5.4]{SS04a}}]\label{lem-defect}
Assume that Hypotheses~\ref{h:so} and~\ref{h:wt} are met. Then there are positive constants $C,\eta>0$ such that
\[
|\bar{u}(x,t) - u_\mathrm{wt}(k_\pm x-t;k_\pm)| \leq C \rme^{-\eta |x|}
\]
uniformly in $(x,t)$.
\end{lemma}

Our last assumption is concerned with nondegeneracy and spectral stability of the source defect $\bar{u}(x,t)$. Consider the linearization
\begin{equation}\label{eqs-lin}
v_t = Dv_{xx} + c_\mathrm{d} v_x + f_u(\bar{u}(x,t)) v
\end{equation}
of (\ref{eqs-RD}) about the source $\bar{u}(x,t)$ and denote by $\Phi_\mathrm{d}:v(x,0)\mapsto v(x,2\pi)$ the associated linear time-$2\pi$ map, where $v(x,t)$ denotes the solution of (\ref{eqs-lin}). We can pose $\Phi_\mathrm{d}$ on $L^2(\mathbb{R})$ and on $L^2_a(\mathbb{R})$, where the latter space is defined by
\[
L^2_a(\mathbb{R}) := \Big\{ u: \rme^{-a|x|} u(x) \in L^2(\mathbb{R}) \Big\}.
\]
Our spectral assumption on the source defect then reads as follows.

\begin{hypothesis}\label{h:ss}
For all sufficiently small $0<a\ll1$, the Floquet spectrum of $\Phi_\mathrm{d}$ on $L^2_a(\mathbb{R})$ lies in the open left half plane and is bounded away from the imaginary axis, with the exception of a Floquet eigenvalue at $\lambda=0$ with geometric and algebraic multiplicity two. The corresponding Floquet eigenfunctions are $\bar{u}_x(x,t)$ and $\bar{u}_t(x,t)$.
\end{hypothesis}

This completes the description of the hypotheses we need in Theorem~\ref{main}, and we now recall a few additional properties of source defects that follow from the hypotheses we made above. Consider the formal $L^2$-adjoint of (\ref{eqs-lin}) given by
\begin{equation}\label{eadj1}
-w_t = Dw_{xx} - c_\mathrm{d} w_x + f_u(\bar{u}(x,t))^* w
\end{equation}
posed on $L^2_{-a}(\mathbb{R})$ with $0<a\ll1$ as in Hypothesis~\ref{h:ss} and denote its period map by $\Phi_\mathrm{d}^\mathrm{adj}$. It follows from Hypothesis~\ref{h:ss} and \cite[Corollary~4.6]{SS04a} that the null space of $\Phi_\mathrm{d}^\mathrm{adj}$ in $L^2_{-a}(\mathbb{R})$ is two-dimensional: we choose two linearly independent $2\pi$-periodic solutions of (\ref{eadj1}) that form a basis of this null space and denote these solutions by $\psi_1(x,t)$ and $\psi_2(x,t)$. Due to the weights in the space $L^2_{-a}(\mathbb{R})$, there exist constants $C,\eta>0$ such that
\begin{equation}\label{eadj2}
|\psi_j(x,t)| \leq C \rme^{-\eta|x|}, \qquad x\in\mathbb{R},\, t\in\mathbb{R}
\end{equation}
for $j=1,2$. Finally, \cite[Corollary~4.6]{SS04a} implies that the matrix $M\in\mathbb{R}^{2\times2}$ given by
\begin{equation}\label{eadj3}
M = \int_\mathbb{R} \int_0^{2\pi} \begin{pmatrix}
\langle \psi_1(x,t),\bar{u}_x(x,t)\rangle_{\mathbb{R}^n} & \langle \psi_1(x,t),\bar{u}_t(x,t)\rangle_{\mathbb{R}^n} \\
\langle \psi_2(x,t),\bar{u}_x(x,t)\rangle_{\mathbb{R}^n} & \langle \psi_2(x,t),\bar{u}_t(x,t)\rangle_{\mathbb{R}^n}
\end{pmatrix}\,\rmd t\, \rmd x
= \begin{pmatrix} 1 & 0 \\ 0 & 1 \end{pmatrix}
\end{equation}
is well defined and invertible (and indeed equal to the $2\times2$ identity matrix).


\section{Laplace transform relates Green's function and resolvent kernel}

As outlined in the introduction, the Green's function $G(x,t;y,s)$ of the linearization
\[
v_t = Lv, \qquad L := D\partial_{x}^2 + c_\mathrm{d} \partial_x + f_u(\bar{u}(x,t))
\]
about the defect $\bar u(x,t)$ is key to our analysis, and we now outline our approach to calculating it. By definition, for each fixed $(y,s)$, the Green's function $G(x,t;y,s)$ satisfies 
\begin{equation}\label{eqs-Gkernel}
(\partial_t - L) G(x,t;y,s) = 0, \qquad
G(x,s;y,s) = \delta(x-y)
\end{equation}
where $\delta(\cdot)$ denotes the standard delta function. The next lemma shows that we can construct $G(x,t; y,s)$ via the Laplace transform.

\begin{lemma}\label{lem-introGlambda} For each $(y,s)$, let $\widetilde{G}_{\lambda}(x,t;y,s)$ be a $2\pi$-periodic function in time $t$, satisfying the following resolvent equation
\begin{equation}\label{eqs-reskernel}
(\lambda + \partial_t - L) \widetilde{G}_{\lambda}(x,t;y,s) =  \delta(x-y) \delta(t-s), 
\end{equation}
for $\lambda \in \mathbb{C}$. Then, the function
\begin{equation}\label{ilt}
G(x,t;y,s) := \frac{1}{2\pi\rmi} \int_{\mu-\frac \rmi2}^{\mu + \frac \rmi2} \rme^{\lambda t} \widetilde{G}_{\lambda}(x,t;y,s) \,\rmd\lambda,
\end{equation}
 for some large constant $\mu>0$, is the Green's function of \eqref{eqs-Gkernel}. 
\end{lemma}
\begin{proof}

Recall that the operator $L$ depends on time through $f_u(\bar u (x,t))$. As $\bar u(x,t)$ is $2\pi$-periodic in time, we can write 
$$ f_u(\bar{u}(x,t)) = \sum_{k\in \mathbb{Z}} \rme^{\rmi k t} a_k(x),$$
for some coefficients $a_k(x)$. For each fixed $y,s$, we set 
$$ G_\lambda (x; y,s) := \int_s^\infty \rme^{-\lambda t} G(x,t;y,s) \; dt ,$$
which is well-defined for all $\lambda\in \mathbb{C}$, with $\Re \lambda >\mu$, for some large constant $\mu>0$. It then follows that $G(x,t;y,s)$ solves \eqref{eqs-Gkernel} if and only if $G_\lambda(x;y,s)$ solves 
$$ (\lambda - D\partial_x^2 - c_\mathrm{d} \partial_x) G_\lambda (x; y,s) - \sum_{k\in \mathbb{Z}} a_k(x) G_{\lambda - \rmi k} (x;y,s)  = e^{-\lambda s} \delta(x-y) .$$
See \cite[Proposition~1.64, Theorem 3.1.3, and \S3.7]{Arendt} for similar arguments. Observe that $G_\lambda(x;y,s)$ couples with $G_{\lambda - \rmi k}(x;y,s)$ for all $k \in \mathbb{Z}$. Thus, to solve the resolvent equations, we are led to introduce 
\[
\widetilde{G}_{\lambda}(x,t;y,s) : = \rme^{\lambda s} \sum_{n\in \mathbb{Z}} e^{in t } G_{\lambda + \rmi n} (x;y,s),
\]
for $\lambda \in \mathbb{C}$. By definition, it suffices to consider $\lambda$ to be in the complex strip: $-\frac12<\mathrm{Im}\, \lambda \le \frac12.$
Observe that $\widetilde{G}_{\lambda}(x,t;y,s) $ is $2\pi$-periodic in time $t$ and satisfies
$$
\begin{aligned} (\lambda  + \partial_t - D\partial_x^2 - c_\mathrm{d} \partial_x) \widetilde{G}_{\lambda}(x,t;y,s) - \rme^{\lambda s}\sum_{n,k\in \mathbb{Z}}  \rme^{\rmi n t}a_k(x)G_{\lambda + \rmi (n-k)} (x;y,s)  =  \delta(t-s)\delta(x-y),
\end{aligned}$$
in which by definition we have $$\rme^{\lambda s}\sum_{n,k\in \mathbb{Z}}  \rme^{\rmi n t}a_k(x)G_{\lambda + \rmi (n-k)} (x;y,s) =  f_u(\bar{u}(x,t))  \widetilde{G}_{\lambda}(x,t;y,s) . $$
This proves that $\widetilde G_\lambda(x,t;y,s)$ solves the resolvent equation \eqref{eqs-reskernel}. The lemma follows, upon inverting the Laplace transform.
\end{proof}

In view of \eqref{ilt}, it suffices to construct the resolvent kernel $\widetilde{G}_{\lambda}(x,t;y,s)$, solving  \eqref{eqs-reskernel} for $\lambda$ such that $|\mathrm{Im}\, \lambda |\le \frac12$. To proceed, we first analyze the homogenous problem: 
\begin{equation}\label{hm-v0}(\lambda + \partial_t - L)v = 0.\end{equation}
Ignoring the fact that this is a parabolic PDE, the key idea is to write the equation as a spatial dynamical system; namely, we introduce 
$$V = \begin{pmatrix} v \\ v_x\end{pmatrix}.$$
The homogenous problem \eqref{hm-v0} becomes
\begin{equation}\label{eqs-Vn} V_x = A(x,\lambda) V, \qquad A(x,\lambda) := \begin{pmatrix} 0 & I \\ D^{-1}[\lambda + \partial_t - f_u(\bar{u}(x,t))] & - c_\mathrm{d}D^{-1} \end{pmatrix}\end{equation}
on the whole line: $x\in\mathbb{R}$. In the next section, we shall construct exponential dichotomies of the dynamical system \eqref{eqs-Vn}. The resolvent kernel $\widetilde{G}_{\lambda}(x,t;y,s)$, solving \eqref{eqs-reskernel}, will then be constructed in Section~\ref{s:resolvent} using the variations-of-constants principle derived in Section~\ref{s:sd}.


\section{Constructing exponential dichotomies}\label{s:sd}

Throughout this section, we assume that Hypotheses~\ref{h:so}, \ref{h:wt}, and \ref{h:ss} are met.

\subsection{Spatial dynamics}\label{s:ed0}

We consider the spatial-dynamics system \eqref{eqs-Vn} given by
\begin{equation}\label{eqs-V}
V_x = A(x,\lambda) V, \qquad A(x,\lambda) := \begin{pmatrix} 0 & I \\ D^{-1}(\lambda + \partial_t - f_u(\bar{u}(x,t))) & - c_\mathrm{d} D^{-1} \end{pmatrix}
\end{equation}
on the Hilbert space 
\begin{equation}\label{spaceYY}
Y  = H^{\frac54}(\mathbb{T},\mathbb{C}^{2n})\times H^{\frac34}(\mathbb{T},\mathbb{C}^{2n}), \qquad \mathbb{T} = [0,2\pi]/\!\sim.
\end{equation}
We denote the norm in $Y$ by $\|\cdot\|_Y$ and record that each element $V$ is $2\pi$-periodic in time and lies in $L^\infty(\mathbb{T})\times L^\infty(\mathbb{T})$. We introduce the linear isomorphism
\[
\mathcal{J}:\quad H^1(\mathbb{T},\mathbb{C}^{2n})\longrightarrow L^2(\mathbb{T},\mathbb{C}^{2n})),\quad v(t)=\sum_{k\in\mathbb{Z}} \hat{v}_k \rme^{\rmi kt} \longmapsto (\mathcal{J}v)(t):=\sum_{k\in\mathbb{Z}} (1+|k|)\hat{v}_k \rme^{\rmi kt}
\]
which allows us to write the scalar product on $Y$ as
\begin{equation}\label{norm-Y}
\langle W,V \rangle_Y
= \left\langle \begin{pmatrix} w_1 \\ w_2 \end{pmatrix}, \begin{pmatrix} v_1 \\ v_2 \end{pmatrix} \right\rangle_Y
= \langle \mathcal{J}^{\frac54}w_1,\mathcal{J}^{\frac54}v_1\rangle_{L^2} + \langle \mathcal{J}^{\frac34}w_2,\mathcal{J}^{\frac34}v_2\rangle_{L^2},
\end{equation}
where we use the scalar product $\langle w,v \rangle_{\mathbb{C}^{2n}}:=\bar{w}^t v$ in $\mathbb{C}^{2n}$. Using this scalar product, we can identify $Y$ with its dual space, and it follows that the system adjoint to \eqref{eqs-V} is again posed on $Y$ and given by
\begin{equation}\label{eqs-Vadj}
W_x = -A(x,\lambda)^* W = - \begin{pmatrix} 0 & \mathcal{J}^{-\frac52}(\bar{\lambda} - \partial_t - f_u(\bar{u}(x,t))^t) D^{-1} \mathcal{J}^{\frac32} \\ \mathcal{J} & - c_\mathrm{d} D^{-1} \end{pmatrix} W;
\end{equation}
see \cite[\S6.2]{SS01} for a similar argument. Note that $A(x,\lambda)^*$ depends analytically on $\bar{\lambda}$.

When $\lambda=0$, the system \eqref{eqs-V} admits the nonzero, linearly independent, bounded solutions
\begin{equation}\label{eV12}
V_1(x) = \begin{pmatrix} \bar{u}_x(x,\cdot) \\ \bar{u}_{xx}(x,\cdot) \end{pmatrix}, \qquad
V_2(x) = \begin{pmatrix} \bar{u}_t(x,\cdot) \\ \bar{u}_{xt}(x,\cdot) \end{pmatrix},
\end{equation}
in $Y$, where $\bar{u}_x$ and $\bar{u}_t$ are the $2\pi$-periodic solutions of \eqref{eqs-lin} introduced in \S\ref{s:prelim}. Similarly, the solutions $\psi_1(x,t)$ and $\psi_2(x,t)$ of the adjoint equation \eqref{eadj1} introduced in \S\ref{s:prelim} generate the nonzero, bounded, linearly independent solutions
\begin{equation}\label{eW12}
W_j(x) = \begin{pmatrix} \mathcal{J}^{-\frac52}(c\psi_j(x,\cdot)-D\partial_x \psi_j(x,\cdot)) \\ \mathcal{J}^{-\frac32}D\psi_j(x,\cdot) \end{pmatrix}, \qquad j=1,2
\end{equation}
of the adjoint system \eqref{eqs-Vadj} at $\lambda=0$ in $Y$. Note that \eqref{eadj2} implies that there are positive constants $C,\eta$ so that
\begin{equation}\label{eW12ed}
|W_j(x)|_Y \leq C \rme^{-\eta |x|}, \qquad x\in\mathbb{R}
\end{equation}
for $j=1,2$. Furthermore, it is easy to check that $\langle W(x),V(x)\rangle_Y$ does not depend on $x$ whenever $W(x)$ satisfies \eqref{eqs-Vadj} and $V(x)$ satisfies \eqref{eqs-V}: we therefore conclude from \eqref{eW12ed} and boundedness of $V_j(x)$ that
\begin{equation}\label{def-W12}
\langle W_i(x), V_j(x) \rangle_Y = 0
\end{equation}
for all $x\in\mathbb{R}$ and each $i,j=1,2$.

\subsection{Exponential dichotomies}\label{s:ed1}

As alluded to earlier, the initial-value problem associated with \eqref{eqs-V} on $Y$ is not well-posed. However, we can solve this system on complementary subspaces either forward or backward in $x$. The following definition encodes this property: we remark that we are typically interested in the case $\kappa^s<0<\kappa^u$ (see Definition~\ref{def-expPhi} below), which guarantees that solutions decay exponentially in the forward or backward $x$-direction.

\begin{definition}[Exponential Dichotomy] \label{def-expPhi} Let $J = \mathbb{R}_+, \mathbb{R}_-$, or $\mathbb{R}$. System \eqref{eqs-V} is said to have an exponential dichotomy on $J$ if there exist constants $K$ and $\kappa^s<\kappa^u$, and two strongly continuous families of bounded operators $\Phi^\mathrm{s}(x,y)$ and $\Phi^\mathrm{u}(x,y)$ on $Y$, defined respectively for $x\ge y$ and $x\le y$, such that $V(x)=\Phi^\mathrm{s}(x,y)V_0$ and $V(x)=\Phi^\mathrm{s}(x,y)V_0$ are solutions of \eqref{eqs-V} for $x>y$ and $x<y$, respectively, for each $V_0\in Y$, the operators $P^\mathrm{s}(x) :=\Phi^\mathrm{s}(x,x)$ and $P^\mathrm{u}(x): = \Phi^\mathrm{u}(x,x)$ are bounded complementary projections in $Y$ for all $x \in J$, and
\[
\sup_{x,y\in J:\, x\geq y} \rme^{-\kappa^s |x-y|} \| \Phi^\mathrm{s}(x,y)\|_{L(Y)} + \sup_{x,y\in J:\, x\leq y} \rme^{\kappa^u |x-y|} \|\Phi^\mathrm{u} (x,y)\|_{L(Y)} \le K,
\]
where $\|\cdot \|_{L(Y)}$ denotes the operator norm of bounded operators from $Y$ to $Y$.
\end{definition}

We emphasize that, if $\Phi^\mathrm{s,u}(x,y)$ define an exponential dichotomy for \eqref{eqs-V} on $J$ with rates $\kappa^s<\kappa^u$, then $\Phi^\mathrm{s}_*(x,y):=\Phi^\mathrm{u}(y,x)^*$ and $\Phi^\mathrm{u}_*(x,y):=\Phi^\mathrm{s}(y,x)^*$ define an exponential dichotomy for \eqref{eqs-Vadj} on $J$ with rates $\kappa^s_\mathrm{ad}:=-\kappa^u<-\kappa^s=:\kappa^u_\mathrm{ad}$; see \cite[Lemma~5.2]{SS01}.

Exponential dichotomies can be used to construct bounded solutions to inhomogeneous systems: the following lemma will be used to later to construct the resolvent kernel.

\begin{lemma}[\cite{PSS}]\label{lem-PhiSolver}
Fix $\lambda\in\mathbb{C}$ and assume that \eqref{eqs-V} has an exponential dichotomy given by $\Phi^\mathrm{s,u}(x,y)$ on $\mathbb{R}$ with constants that satisfy $\kappa^s<0<\kappa^u$. For each $F\in C^0(\mathbb{R};Y)$, the system
\[
V_x = A(x,\lambda) V + F(x)
\]
then has a unique bounded solution, and this solution is given by the variation-of-constants formula
\[
V(x) = \int_{-\infty}^x \Phi^\mathrm{s}(x,y) F(y)\; \rmd y  + \int_\infty^x \Phi^\mathrm{u} (x,y) F(y)\; \rmd y, \qquad x\in\mathbb{R}.
\]
\end{lemma}

We now comment on cases when \eqref{eqs-V} has exponential dichotomies.

It was shown in \cite[Corollary~A.2]{SS04a} that \eqref{eqs-V} has an exponential dichotomy on $\mathbb{R}_+$ with $\kappa^s<0<\kappa^u$ if, and only if, the system
\[
V_x = \begin{pmatrix} 0 & I \\ D^{-1}(\lambda + \partial_t - f_u(\bar{u}_\mathrm{wt}(k_+x-t;k_+))) & - c_\mathrm{d} D^{-1} \end{pmatrix}
\]
for the asymptotic wave train at $x=\infty$ has an exponential dichotomy on $\mathbb{R}$ with $\kappa^s<0<\kappa^u$ or, equivalently, when $\lambda$ does not belong to the Floquet spectrum of the asymptotic wave train; the same statement holds for exponential dichotomies on $\mathbb{R}_-$ upon using the asymptotic wave train at $x=-\infty$. It follows from \cite[\S4.2]{SS04a} that \eqref{eqs-V} has an exponential dichotomy on $\mathbb{R}$ with $\kappa^s<0<\kappa^u$ if, and only if, it has exponential dichotomies $\Phi^\mathrm{s,u}_\pm(x,y)$ on $\mathbb{R}_\pm$ with $\kappa^s<0<\kappa^u$ that satisfy $\Rg(P^\mathrm{s}_+(0))\oplus\Rg(P^\mathrm{u}_-(0))=Y$.

Hypotheses~\ref{h:so}, \ref{h:wt}, and~\ref{h:ss} together with \cite[Corollary~A.2]{SS04a} imply that \eqref{eqs-V} has an exponential dichotomy on $J=\mathbb{R}$ with $\kappa^s<0<\kappa^u$ for each $\lambda$ with $\Re\lambda\geq0$ except at $\lambda=0$. Furthermore, these dichotomies are analytic in $\lambda$ in the right half-plane. There are two ways in which the existence of exponential dichotomies with $\kappa^s<0<\kappa^u$ breaks down at $\lambda=0$: Firstly, Hypothesis~\ref{h:ss} implies that the nonzero functions $V_{1,2}(x)$ given in \eqref{eV12} satisfy \eqref{eqs-V} at $\lambda=0$, which precludes the existence of exponential dichotomies on $\mathbb{R}$ for any $\kappa^s<\kappa^u$. Secondly, Lemma~\ref{lem-ss} describes the Floquet spectra of the asymptotic wave trains and, in particular, shows that they both contain $\lambda=0$, thus precluding the existence of exponential dichotomies on $\mathbb{R}_\pm$ with $\kappa^s<0<\kappa^u$ based on the criteria from \cite[Corollary~A.2]{SS04a} we just reviewed.

In the next section, we will show that we can define exponential dichotomies for \eqref{eqs-V} on $\mathrm{R}_\pm$ for certain rates $\kappa^s<\kappa^u$ for each $\lambda$ near zero, and we will also show that these dichotomies can be constructed in such a way that they depend analytically on $\lambda$.

\subsection{Constructing analytic exponential dichotomies on $\mathbb{R}_\pm$ near $\lambda=0$}\label{s:ed2}

\begin{figure}[b]
\centering
\includegraphics[scale=0.7]{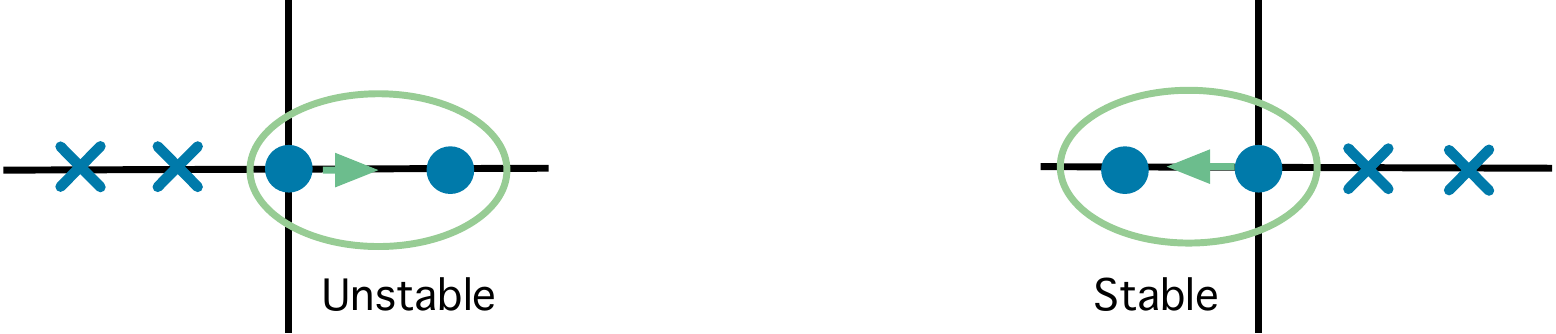}
\caption{Shown is an illustration of the spatial Floquet exponents of the asymptotic wave trains at $x=-\infty$ (left) and $x=\infty$ (right) at $\lambda=0$. The spatial Floquet exponents admit the expansion $\nu_\pm(\lambda)=-\lambda/c_\pm+\rmO(\lambda^2)$ with $c_-<0<c_+$ and therefore move in the direction indicated by the arrows as $\lambda$ moves into the open right half-plane.}
\label{fig-spatialODE}
\end{figure}

We begin by considering the asymptotic systems
\begin{equation}\label{eqs-wt}
V_x = A_\pm(x,\lambda) V, \qquad
A(x,\lambda) := \begin{pmatrix} 0 & I \\ D^{-1}(\lambda + \partial_t - f_u(u_\mathrm{wt}(k_\pm x-t))) & - c_\mathrm{d} D^{-1} \end{pmatrix}
\end{equation}
associated with \eqref{eqs-V}, where $u_\mathrm{wt}(k_\pm x-t)$ denotes the asymptotic wave trains that the defect $\bar{u}(x,t)$ converges to as $x\to\pm\infty$. Note that these systems are periodic in $x$, and the \cite{SS04a,SS08} imply that they have well-defined Floquet exponents. We now summarize some of the consequences of the results in \cite[\S3.4 and \S4]{SS04a} in conjunction with Lemma~\ref{lem-ss}. Firsty, the Floquet exponents of \eqref{eqs-wt} are uniformly bounded away from the imaginary axis for $\Re\lambda\geq0$ except near $\lambda=0$. For $\lambda=0$, each of the systems \eqref{eqs-wt}$_\pm$ has precisely one Floquet exponent at the origin, with the remaining Floquet exponents being uniformly bounded away from the imaginary axis. The Floquet exponent of \eqref{eqs-wt}$_\pm$ that lies at the origin for $\lambda=0$ varies analytically in $\lambda$ and has the expansion
\begin{equation}\label{def-0eigenvalues}
\nu_\pm(\lambda) = - \frac{\lambda}{c_\pm} + \rmO(\lambda^2),
\end{equation}
where the constants $c_\pm$  are the group velocities of the asymptotic wave trains, which satisfy the inequality $c_-<0<c_+$ due to Hypothesis~\ref{h:wt} (see also Figure~\ref{fig-spatialODE}). We also note that the adjoint systems belonging to \eqref{eqs-wt}, which are given by \eqref{eqs-Vadj} with $\bar{u}$ replaced by $u_\mathrm{wt}$, admit unique Floquet exponents $-\overline{\nu_\pm(\lambda)}=-\nu_\pm(\bar{\lambda})$ near $\lambda=0$.

Next, we define the two subspaces
\[
E_0^\mathrm{pt} := \mathrm{span}\{V_1(0),V_2(0)\}, \qquad
E_0^\mathrm{ad} := \mathrm{span}\{W_1(0),W_2(0)\}
\]
of $Y$ and note that $E_0^\mathrm{pt}\perp E_0^\mathrm{ad}$ due to \eqref{def-W12}. We can now state the following result on the existence of exponential dichotomies of \eqref{eqs-V} on $\mathbb{R}_\pm$ for $\lambda$ near zero.

\begin{lemma}\label{lem-Phipm}
Assume that Hypotheses~\ref{h:so}, \ref{h:wt}, and \ref{h:ss} are met. There are then rates $0<\kappa^s_+<\kappa^u_+$ such that the system \eqref{eqs-V} has an exponential dichotomy $\Phi^\mathrm{s,u}_+(x,y,\lambda)$ on $\mathbb{R}_+$ for all $\lambda$ near zero. The operators $\Phi^\mathrm{s,u}_+(x,y,\lambda)$ depend analytically on $\lambda$, and there is a constant $\eta>0$ and solutions $V^c_+(x,\lambda)$ of \eqref{eqs-V} and $W^c_+(x,\lambda)$ of \eqref{eqs-Vadj} that are analytic in $\lambda$ and $\bar{\lambda}$, respectively, so that
\begin{eqnarray*}
\Phi_+^\mathrm{s}(x,y,\lambda) & = & \rme^{\nu_+^c(\lambda)(x-y)} V_+^c(x,\lambda) \langle W_+^c(y,\lambda),\cdot\rangle_Y + \rmO(\rme^{-\eta |x-y|}), \qquad x\geq y\geq0 \\
\Phi_+^\mathrm{u}(x,y,\lambda) & = & \rmO(\rme^{-\eta |x-y|}), \qquad y\geq x\geq0,
\end{eqnarray*}
where the $\rmO(\cdot)$ terms are bounded operators that are analytic in $\lambda$.

Similarly, there are rates $\kappa^s_-<\kappa^u_-<0$ such that the system \eqref{eqs-V} has an exponential dichotomy $\Phi^\mathrm{s,u}_-(x,y,\lambda)$ on $\mathbb{R}_-$ for all $\lambda$ near zero, and we have the expansion
\begin{eqnarray*}
\Phi_-^\mathrm{u}(x,y,\lambda) & = & \rme^{\nu_-^c(\lambda)(x-y)} V_-^c(x,\lambda) \langle W_-^c(y,\lambda),\cdot\rangle_Y + \rmO(\rme^{-\eta |x-y|}), \qquad x\leq y\leq0 \\
\Phi_-^\mathrm{s}(x,y,\lambda) & = & \rmO(\rme^{-\eta |x-y|}), \qquad y\leq x\leq0,
\end{eqnarray*}
where all terms have analyticity properties analogous to those of the dichotomies on $\mathbb{R}_+$. 

Finally, $V^c_\pm(x,\lambda)$ can be chosen so that $V^c_\pm(0,0)\in E_0^\mathrm{pt}$. Furthermore, the exponential dichotomies can be chosen so that there are closed subspaces $E_0^\mathrm{s,u}$ of $Y$ with
\begin{eqnarray}\label{e1}
\Rg(\Phi^\mathrm{s}_+(0,0,0)) = E_0^\mathrm{pt}\oplus E_0^\mathrm{s},
& \quad &
\Rg(\Phi^\mathrm{u}_-(0,0,0)) = E_0^\mathrm{pt}\oplus E_0^\mathrm{u},
\\ \nonumber
\Rg(\Phi^\mathrm{u}_+(0,0,0)) = E_0^\mathrm{u} \oplus E_0^\mathrm{ad},
& \quad &
\Rg(\Phi^\mathrm{s}_-(0,0,0)) = E_0^\mathrm{s} \oplus E_0^\mathrm{ad},
\end{eqnarray}
where $E_0^\mathrm{pt}\oplus E_0^\mathrm{s}\oplus E_0^\mathrm{u}\oplus E_0^\mathrm{ad}=Y$.
\end{lemma}

\begin{proof}
The proof is similar to the proofs in \cite[Lemma~4]{BSZ} and \cite{Hupkes}, and we therefore provide only a brief summary of the strategy of the proof. Using relative Morse indices and our hypotheses, we can proceed as \cite{SS04a,SS08} to use exponential weights to conclude the existence of analytic exponential dichotomies that belong to the rates $0<\kappa^s_+<\kappa^u_+$, vary analytically in $\lambda$ near $\lambda=0$, and satisfy \eqref{e1}. It therefore remains to derive the expansion for $\Phi^\mathrm{s}_+(x,y,\lambda)$ and to prove the the last part of the lemma.

Recall that the asymptotic system at $x=\infty$ has a unique Floquet exponent at the origin for $\lambda=0$. Lemma~\ref{lem-defect} implies that $\|A(x,\lambda)-A_\pm(x,\lambda)\|_{L(Y)}$ converges to zero exponentially as $x\to\pm\infty$ with rate independent of $\lambda$. We can therefore use the gap lemma \cite{GZ,KS,SS04a,SS08,BSZ} to construct a solution $V^c_+(x,\lambda)$ of \eqref{eqs-V} that satisfies $V^c_+(0,0)\in E_0^\mathrm{pt}$, is analytic in $\lambda$, and converges exponentially to a nonzero Floquet eigenfunction of the asymptotic system \eqref{eqs-wt}$_+$ belonging to the Floquet exponent $\nu_\pm(\lambda)$ for each $\lambda$.

Similarly, we can construct a solution $W^c_+(x,\lambda)$ of \eqref{eqs-Vadj} that converges exponentially to a nonzero Floquet eigenfunction belonging to the Floquet exponent $-\nu_+(\bar{\lambda})$ of the system adjoint to asymptotic system \eqref{eqs-wt}$_+$ as $x\to\infty$, is analytic in $\bar{\lambda}$, and satisfies $W^c_+(0,\lambda)\perp \Rg(\Phi^\mathrm{u}_+(0,0,\lambda)$ for each $\lambda$ near zero.

Using exponential weights $-1\ll\tilde{\kappa}^s_+<\tilde{\kappa}^u_+<0$ that separate the Floquet exponent $\nu_+(\lambda)$ from the remaining stable Floquet exponents of \eqref{eqs-wt}, we can then construct analytic strong stable dichotomies $\Phi^\mathrm{ss}_+(x,y,\lambda)$ of \eqref{eqs-V} on $\mathbb{R}_+$. In summary, we arrive at the decomposition
\begin{equation}\label{e2}
\Phi_+^\mathrm{s}(x,y,\lambda) = \rme^{\nu_+^c(\lambda)(x-y)} V_+^c(x,\lambda) \langle W_+^c(y,\lambda),\cdot\rangle_Y + \Phi_+^\mathrm{ss}(x,y,\lambda), \qquad x\geq y\geq0
\end{equation}
with $\|\Phi_+^\mathrm{ss}(x,y,\lambda)\|\leq C\rme^{-\eta|x-y|}$. Since the rates we used above to construct $\Phi_+^\mathrm{s,u}$ satisfy $0<\kappa^s_+<\kappa^u_+$, the complementary dichotomy $\Phi_+^\mathrm{u}(x,y,\lambda)$ also satisfies $\|\Phi_+^\mathrm{u}(x,y,\lambda)\|\leq C\rme^{-\eta|x-y|}$ for each $0<\eta\leq\kappa^u_+$ as claimed. Note that the individual terms on the right-hand side of the decomposition \eqref{e2} of $\Phi_+^\mathrm{s}(x,y,\lambda)$ are analytic: in particular, the term $\langle W_+^c(y,\lambda),\cdot\rangle_Y$ is analytic in $\lambda$ since $W_+^c(y,\lambda)$ is analytic in $\bar{\lambda}$ and we use the scalar product $\langle w,v \rangle_{\mathbb{C}^{2n}}=\bar{w}^t v$ in $\mathbb{C}^{2n}$ in our definition of $\langle\cdot,\cdot\rangle_Y$.

The construction of analytic exponential dichotomies on $\mathbb{R}_-$ is similar, and it therefore remains to prove the assertions in \eqref{e1}. Since the closed subspace $\Rg(\Phi^\mathrm{s}_+(0,0,0))$ consists of all elements $V_0$ in $Y$ for which there is a solution $V(x)$ of \eqref{eqs-V} at $\lambda=0$ that satisfies $V(0)=V_0$ and $|V(x)|\leq C\rme^{\kappa^s_+ x}$ for $x\geq0$, we see that this subspace is unique and contains $V_j(0)$ (recall that $\kappa^s_+>0$). A similar argument holds for $\Rg(\Phi^\mathrm{u}_-(0,0,0))$. Next, we claim that
\[
\Rg(\Phi^\mathrm{s}_+(0,0,0)) \cap \Rg(\Phi^\mathrm{u}_-(0,0,0)) = E_0^\mathrm{pt}, \quad
\left(\Rg(\Phi^\mathrm{s}_+(0,0,0)) + \Rg(\Phi^\mathrm{u}_-(0,0,0))\right)^\perp = E_0^\mathrm{ad}.
\]
Indeed, the first equation follows from Hypothesis~\ref{h:ss} upon using the exponential weights for the exponential dichotomies on $\mathbb{R}_\pm$. The second assertion is then a consequence of the Fredholm properties proved in \cite[Lemma~4.2]{SS04a} together with the characterization of the adjoint operator shown in \cite[Lemma~6.1]{SS01}. It remains to prove that we can construct the exponential dichotomies so that
\[
\Rg(\Phi^\mathrm{u}_+(0,0,0)) = E_0^\mathrm{u} \oplus E_0^\mathrm{ad}, \quad
\Rg(\Phi^\mathrm{s}_-(0,0,0)) = E_0^\mathrm{s} \oplus E_0^\mathrm{ad},
\]
which is a consequence of \cite[(3.20)]{PSS}. This completes the proof of the lemma.
\end{proof}

\subsection{Extending exponential dichotomies meromorphically to $\mathbb{R}$ near $\lambda=0$}\label{s:ed3}

As mentioned in \S\ref{s:ed1}, our hypotheses together with \cite[Corollary~A.2]{SS04a} imply that \eqref{eqs-V} has an exponential dichotomy on $J=\mathbb{R}$ with $\kappa^s<0<\kappa^u$ for each $\lambda$ with $\Re\lambda\geq0$ except at $\lambda=0$. Furthermore, these dichotomies are analytic in $\lambda$ in the right half-plane. In \S\ref{s:ed2}, we restricted these dichotomies to $\mathbb{R}_+$ and $\mathbb{R}_-$ and extended each restriction to an open neighborhood of $\lambda=0$ so that the resulting exponential dichotomies on $\mathbb{R}_\pm$ with rates $0<\kappa^\mathrm{s}_+<\kappa^\mathrm{u}_+$ on $\mathbb{R}_+$ and $\kappa^\mathrm{s}_-<\kappa^\mathrm{u}_-<0$ on $\mathbb{R}_-$ vary analytically in $\lambda$. At $\lambda=0$, the extended dichotomies satisfy
\[
\Rg(\Phi^\mathrm{s}_+(0,0,0))\cap\Rg(\Phi^\mathrm{u}_-(0,0,0))=E_0^\mathrm{pt},
\]
which shows that we cannot expect an exponential dichotomy to exist at $\lambda=0$.

In this section, we will construct solution operators $\Phi^\mathrm{s}(x,y,\lambda)$, defined for $y\leq x$, and $\Phi^\mathrm{u}(x,y,\lambda)$, defined for $x\leq y$, of \eqref{eqs-V} that are meromorphic in $\lambda$ in a neighborhood of the origin with a simple pole at $\lambda=0$. The solution operators $\Phi^\mathrm{s,u}(x,y,\lambda)$ coincide with the exponential dichotomies when $\Re\lambda>0$ and satisfy
\[
\Rg(\Phi^\mathrm{s}(0,0,\lambda)) = \Rg(\Phi^\mathrm{s}_+(0,0,\lambda)), \quad
\Rg(\Phi^\mathrm{u}(0,0,\lambda)) = \Rg(\Phi^\mathrm{u}_-(0,0,\lambda))
\]
for all $\lambda\neq0$. We use a strategy similar to the one used in \cite{BSZ,Hupkes}. Let
\[
E_+^\mathrm{s,u}(\lambda) := \Rg(\Phi_+^\mathrm{s,u}(0,0,\lambda)), \quad
E_-^\mathrm{s,u}(\lambda) := \Rg(\Phi_-^\mathrm{s,u}(0,0,\lambda)),
\]
then the key to constructing the desired solution operators is to write $E_-^\mathrm{u}(\lambda)$ as the graph of a bounded operator $h_+(\lambda):E_+^\mathrm{u}(\lambda)\to E_+^\mathrm{s}(\lambda)$ that is meromorphic in $\lambda$ and, similarly, represent $E_+^\mathrm{s}(\lambda)$ as the graph of a bounded operator $h_-(\lambda):E_-^\mathrm{s}(\lambda)\to E_-^\mathrm{u}(\lambda)$.

Note that, by construction, $E_+^\mathrm{s}(\lambda)\oplus E_+^\mathrm{u}(\lambda)=Y$ and $E_-^\mathrm{s}(\lambda)\oplus E_-^\mathrm{u}(\lambda)=Y$ for all $\lambda$ near zero. Recall from Lemma~\ref{lem-Phipm} that
\begin{equation}\label{e4}
E^\mathrm{s}_+(0) = E_0^\mathrm{pt}\oplus E_0^\mathrm{s}, \quad
E^\mathrm{u}_+(0) = E_0^\mathrm{u} \oplus E_0^\mathrm{ad}, \quad
E^\mathrm{u}_-(0) = E_0^\mathrm{pt}\oplus E_0^\mathrm{u}, \quad
E^\mathrm{s}_-(0) = E_0^\mathrm{s} \oplus E_0^\mathrm{ad}
\end{equation}
where
\[
E_0^\mathrm{pt}\oplus E_0^\mathrm{s}\oplus E_0^\mathrm{u}\oplus E_0^\mathrm{ad}=Y.
\]
We then have the following result.

\begin{lemma}\label{lem-EEE}
For each $\lambda$ near the origin, there are unique bounded operators
\[
g^\mathrm{s}_+(\lambda): E_0^\mathrm{pt}\oplus E_0^\mathrm{s} \to E_0^\mathrm{u}\oplus E_0^\mathrm{ad}, \quad
g^\mathrm{u}_-(\lambda): E_0^\mathrm{pt}\oplus E_0^\mathrm{u} \to E_0^\mathrm{s}\oplus E_0^\mathrm{ad}, \quad
g^\mathrm{u}_+(\lambda): E_0^\mathrm{u}\oplus E_0^\mathrm{ad} \to E_0^\mathrm{s}\oplus E_0^\mathrm{pt}
\]
such that
\[
E_+^\mathrm{s}(\lambda) = \graph \lambda g^\mathrm{s}_+(\lambda), \quad
E_-^\mathrm{u}(\lambda) = \graph \lambda g^\mathrm{u}_-(\lambda), \quad
E_+^\mathrm{u}(\lambda) = \graph \lambda g^\mathrm{u}_+(\lambda).
\]
Furthermore, these operators are analytic in $\lambda$ for $\lambda$ near the origin.
\end{lemma}

\begin{proof}
The subspaces $E_\pm^\mathrm{s,u}(\lambda)$ depend analytically on $\lambda$, and the result therefore follows from \eqref{e4}.
\end{proof}

Let $P_0^\mathrm{ad}$ be the projection onto $E_0^\mathrm{ad}$ with null space $E_0^\mathrm{pt}\oplus E_0^\mathrm{s}\oplus E_0^\mathrm{u}$. We have the following result for the difference $P_0^\mathrm{ad}(g^\mathrm{u}_-(0)-g^\mathrm{s}_+(0))$ restricted to $E_0^\mathrm{pt}$.

\begin{lemma}\label{lem-invertible}
The operator
\[
P_0^\mathrm{ad} \left(g^\mathrm{u}_-(0) - g^\mathrm{s}_+(0)\right)\Big|_{E_0^\mathrm{pt}}:\quad E_0^\mathrm{pt} \longrightarrow E_0^\mathrm{ad}
\]
is invertible and given by the matrix $M$ from \eqref{eadj3} with respect to the bases $\{V_1(0),V_2(0)\}$ of $E_0^\mathrm{pt}$ and $\{W_1(0),W_2(0)\}$ of $E_0^\mathrm{ad}$.
\end{lemma}

\begin{proof} 
Recall that the graph of $\lambda g^\mathrm{s}_+(\lambda)$ is the space $\Rg(\Phi^\mathrm{s}_+(0,0,\lambda))$ and, similarly, the graph of $\lambda g^\mathrm{u}_-(\lambda)$ is the space $\Rg(\Phi^\mathrm{u}_-(0,0,\lambda))$. Hence, $P_0^\mathrm{ad}(g^\mathrm{u}_-(0)-g^\mathrm{s}_+(0))$ restricted to $E_0^\mathrm{pt}$ is represented by the matrix with entries
\[
\langle W_i(0),\partial_\lambda(\Phi^\mathrm{u}_-(0,0,\lambda)-\Phi^\mathrm{s}_+(0,0,\lambda))|_{\lambda=0} V_j(0) \rangle_Y.
\]
It follows from the construction of exponential dichotomies in \cite{PSS,SS01,SS04a} that
\[
\Phi^\mathrm{s}_+(0,0,\lambda) = \Phi^\mathrm{s}_+(0,0,0)
+ \lambda \int_\infty^0 \Phi^\mathrm{u}_+(0,x,0) B \Phi^\mathrm{s}_+(x,0,\lambda) \,\rmd x, \qquad
B = \begin{pmatrix} 0 & 0 \\ D^{-1} & 0 \end{pmatrix}.
\]
In particular, we have
\begin{eqnarray*}
\lefteqn{\langle W_i(0),\partial_\lambda\Phi^\mathrm{s}_+(0,0,\lambda)|_{\lambda=0} V_j(0)\rangle_Y} \\ & = &
\int_\infty^0 \langle W_i(0),\Phi^\mathrm{u}_+(0,x,0) B \Phi^\mathrm{s}_+(x,0,0) V_j(0) \rangle_Y \,\rmd x = 
\int_\infty^0 \langle W_i(x),B V_j(x) \rangle_Y \,\rmd x,
\end{eqnarray*}
where we used that $V_j(x)=\Phi^\mathrm{s}_+(x,0,0)V_j(0)$, $W_j(x)=\Phi^\mathrm{u}_+(0,x,0)^*W_i(0)$, and \eqref{def-W12}. Proceeding in the same way for $\lambda g^\mathrm{u}_-(\lambda)$ and using the definitions of $V_j(x)$, $W_i(x)$, and $\langle\cdot,\cdot,\rangle_Y$ from \S\ref{s:ed0}, we see that $P_0^\mathrm{ad}(g^\mathrm{u}_-(0)-g^\mathrm{s}_+(0))$ restricted to $E_0^\mathrm{pt}$ has the matrix representation
\begin{eqnarray}\label{e6}
\lefteqn{\int_\mathbb{R} \begin{pmatrix}
\langle W_1(x),BV_1(x)\rangle_Y & \langle W_1(x),BV_2(x)\rangle_Y \\
\langle W_2(x),BV_1(x)\rangle_Y & \langle W_2(x),BV_2(x)\rangle_Y
\end{pmatrix}\, \rmd x}
\\ \nonumber & = &
\int_\mathbb{R} \begin{pmatrix}
\langle \mathcal{J}^{\frac34} \mathcal{J}^{-\frac32}D\psi_1(x,\cdot),\mathcal{J}^{\frac34}D^{-1}\bar{u}_x(x,\cdot) \rangle_{L^2} &
\langle \mathcal{J}^{\frac34} \mathcal{J}^{-\frac32}D\psi_1(x,\cdot),\mathcal{J}^{\frac34}D^{-1}\bar{u}_t(x,\cdot) \rangle_{L^2} \\
\langle \mathcal{J}^{\frac34} \mathcal{J}^{-\frac32}D\psi_2(x,\cdot),\mathcal{J}^{\frac34}D^{-1}\bar{u}_x(x,\cdot) \rangle_{L^2} &
\langle \mathcal{J}^{\frac34} \mathcal{J}^{-\frac32}D\psi_2(x,\cdot),\mathcal{J}^{\frac34}D^{-1}\bar{u}_t(x,\cdot) \rangle_{L^2}
\end{pmatrix}\, \rmd x
\\ \nonumber & = &
\int_\mathbb{R} \int_0^{2\pi} \begin{pmatrix}
\langle \psi_1(x,t),\bar{u}_x(x,t)\rangle_{\mathbb{R}^n} & \langle \psi_1(x,t),\bar{u}_t(x,t)\rangle_{\mathbb{R}^n} \\
\langle \psi_2(x,t),\bar{u}_x(x,t)\rangle_{\mathbb{R}^n} & \langle \psi_2(x,t),\bar{u}_t(x,t)\rangle_{\mathbb{R}^n}
\end{pmatrix}\,\rmd t\, \rmd x
\end{eqnarray}
with respect to the bases $\{V_1(0),V_2(0)\}$ of $E_0^\mathrm{pt}$ and $\{W_1(0),W_2(0)\}$ of $E_0^\mathrm{ad}$. We showed at the end of~\S\ref{s:prelim} that the matrix $M$ on the right-hand side of \eqref{e6} is indeed invertible as claimed.
\end{proof}

Let $P_0^\mathrm{u,ad}$ be the projection onto $E_0^\mathrm{u}\oplus E_0^\mathrm{ad}$ with null space $E_0^\mathrm{pt}\oplus E_0^\mathrm{s}$. We can then state our key result.

\begin{lemma}\label{lem:mer-ext}
There is a $\delta>0$ such that, for each $\lambda\in U_\delta(0)\setminus\{0\}$, there is a unique linear bounded operator $h^+(\lambda): E_+^\mathrm{u}(\lambda)\to E_+^\mathrm{s}(\lambda)$ so that $E_-^\mathrm{u}(\lambda)=\graph h^+(\lambda)$, and we have $h^+(\lambda)=\tilde{h}^+(\lambda)P_0^\mathrm{u,ad}$ with
\[
\tilde{h}^+(\lambda) = \frac{1}{\lambda} M^{-1} P_0^\mathrm{ad} + h^+_a(\lambda) : E_0^\mathrm{u}\oplus E_0^\mathrm{ad}\longrightarrow Y, \quad
\Rg(\tilde{h}_+(\lambda)) \subset E^\mathrm{s}_+(\lambda)
\]
where $h^+_a(\lambda)$ is analytic for $\lambda\in U_\delta(0)$. Similarly, for each $\lambda\in U_\delta(0)\setminus\{0\}$, there is a unique linear bounded operator $h^-(\lambda): E_-^\mathrm{s}(\lambda)\to E_-^\mathrm{u}(\lambda)$ so that $E_+^\mathrm{s}(\lambda)=\graph h^-(\lambda)$ and we have $h^-(\lambda)=\tilde{h}^-(\lambda)P_0^\mathrm{s,ad}$ with
\[
\tilde{h}^-(\lambda) = -\frac{1}{\lambda} M^{-1} P_0^\mathrm{ad} + h^-_a(\lambda): E_0^\mathrm{s}\oplus E_0^\mathrm{ad}\longrightarrow Y, \quad
\Rg(\tilde{h}_+(\lambda)) \subset E^\mathrm{u}_-(\lambda)
\]
where $h^-_a(\lambda)$ is analytic for $\lambda\in U_\delta(0)$. 
\end{lemma}

\begin{proof}
For each $V\in Y$, we write
\begin{equation}\label{e10}
V = (V^\mathrm{pt},V^\mathrm{s},V^\mathrm{u},V^\mathrm{ad}) \in E_0^\mathrm{pt} \oplus E_0^\mathrm{s} \oplus E_0^\mathrm{u} \oplus E_0^\mathrm{ad}.
\end{equation}
We focus on $h^+(\lambda)$ as the proof for $h^-(\lambda)$ is analogous. First, we note that $E^\mathrm{u}_+(\lambda)\oplus E^\mathrm{s}_+(\lambda)=Y$, and we can therefore write any element $\widetilde{V}$ of $E^\mathrm{u}_-(\lambda)$ uniquely as the sum of elements in $E^\mathrm{u}_+(\lambda)$ and $E^\mathrm{s}_+(\lambda)$, Using Lemma~\ref{lem-EEE} then results in the system
\begin{equation}\label{e11}
\underbrace{\widetilde V^\mathrm{u} + \widetilde V^\mathrm{pt} + \lambda g^\mathrm{u}_-(\lambda) (\widetilde V^\mathrm{u} + \widetilde V^\mathrm{pt})}_{\in E_-^\mathrm{u}(\lambda)} =
\underbrace{V^\mathrm{u} + V^\mathrm{ad} + \lambda g^\mathrm{u}_+(\lambda) (V^\mathrm{u} + V^\mathrm{ad})}_{\in E_+^\mathrm{u}(\lambda)} + \underbrace{ V^\mathrm{s} + V^\mathrm{pt} + \lambda g^\mathrm{s}_+(\lambda) (V^\mathrm{s} + V^\mathrm{pt})}_{\in E_+^\mathrm{s}(\lambda)},
\end{equation}
where the operators on the right-hand side are analytic in $\lambda$ for $\lambda\in U_\delta(0)$. For each $\widetilde V\in E^\mathrm{u}_-(\lambda)$, we need to express $(V^\mathrm{s},V^\mathrm{pt})$ in terms of $(V^\mathrm{u},V^\mathrm{ad})$, which will then determine the desired operator $h^+(\lambda)$. We decompose \eqref{e11} into the components given in \eqref{e10} and arrive at the system
\begin{eqnarray*}
\widetilde V^\mathrm{u} & = & V^\mathrm{u} + \lambda P^\mathrm{u}_0 g^\mathrm{s}_+(\lambda)(V^\mathrm{s}+V^\mathrm{pt}) \\
\widetilde V^\mathrm{pt} & = & V^\mathrm{pt} + \lambda P^\mathrm{pt}_0 g^\mathrm{u}_+(\lambda)(V^\mathrm{u}+V^\mathrm{ad}) \\
V^\mathrm{ad} & = & \lambda P^\mathrm{ad}_0\left( g^\mathrm{u}_-(\lambda) (\widetilde V^\mathrm{u} + \widetilde V^\mathrm{pt}) - g^\mathrm{s}_+(\lambda) (V^\mathrm{s} + V^\mathrm{pt}) \right) \\
V^\mathrm{s} & = & \lambda P^\mathrm{s}_0 \left( g^\mathrm{u}_-(\lambda) (\widetilde V^\mathrm{u} + \widetilde V^\mathrm{pt}) -g^\mathrm{u}_+(\lambda) (V^\mathrm{u}+V^\mathrm{ad}) \right).
\end{eqnarray*}
The first two equations uniquely determine $(\widetilde V^\mathrm{u},\widetilde V^\mathrm{pt})$ in terms of $V$ and substituting these expressions into the remaining two equations gives
\begin{eqnarray}
V^\mathrm{ad} & = & \lambda P^\mathrm{ad}_0 \left(g^\mathrm{u}_-(\lambda)(V^\mathrm{u}+V^\mathrm{pt}) - g^\mathrm{s}_+(\lambda) (V^\mathrm{s} + V^\mathrm{pt}) + \lambda h_1(\lambda) V\right) \label{e12} \\ \nonumber
V^\mathrm{s} & = & \lambda h_2(\lambda) V
\end{eqnarray}
for certain bounded operators $h_{1,2}(\lambda)$ that are analytic in $\lambda\in U_\delta(0)$. In particular, there is a bounded operator $h_3(\lambda)$ that depends analytically on $\lambda\in U_\delta(0)$ so that
\begin{equation}\label{e13}
V^\mathrm{ad} + V^\mathrm{s} = \lambda h_3(\lambda) (V^\mathrm{u}+V^\mathrm{pt}).
\end{equation}
It remains to express $V^\mathrm{pt}$ in terms of $V^\mathrm{u}+V^\mathrm{ad}$. To accomplish this, we substitute \eqref{e13} into the equation for $V^\mathrm{ad}$ in \eqref{e12} and obtain
\begin{equation}\label{e14}
V^\mathrm{ad} = \lambda \left( P^\mathrm{ad}_0 (g^\mathrm{u}_-(0)-g^\mathrm{s}_+(0)) + \lambda h_5(\lambda) \right) V^\mathrm{pt} + \lambda h_4(\lambda) V^\mathrm{u}
\end{equation}
for certain bounded operators $h_{4,5}(\lambda)$ that are analytic in $\lambda\in U_\delta(0)$. Since $M=P^\mathrm{ad}_0(g^\mathrm{u}_-(0)-g^\mathrm{s}_+(0))|_{E_0^\mathrm{pt}}$ is invertible by Lemma~\ref{lem-invertible}, we have
\[
\left(M + \lambda h_5(\lambda)\right)^{-1} = M^{-1} + \lambda h_6(\lambda)
\]
for a certain bounded operator $h_6(\lambda)$ that is analytic in $\lambda\in U_\delta(0)$. For each $\lambda\in U_\delta(0)\setminus\{0\}$, we can therefore write \eqref{e14} as
\[
\frac{1}{\lambda} (M^{-1}+\lambda h_6(\lambda)) V^\mathrm{ad} = V^\mathrm{pt} + (M^{-1}+\lambda h_6(\lambda)) h_4(\lambda) V^\mathrm{u}
\]
and hence obtain that
\[
V^\mathrm{pt} = \frac{1}{\lambda} M^{-1} V^\mathrm{ad} + h_7(\lambda) (V^\mathrm{ad}+V^\mathrm{u}),
\]
where the bounded operator $h_7(\lambda)$ is analytic in $\lambda\in U_\delta(0)$. Recalling \eqref{e13}, we finally arrive at the representation
\[
V^\mathrm{pt} + V^\mathrm{s} = \frac{1}{\lambda} M^{-1} V^\mathrm{ad} + \tilde{h}^+_\mathrm{a}(\lambda) (V^\mathrm{ad}+V^\mathrm{u})
\]
that is valid for all $\lambda\in U_\delta(0)\setminus\{0\}$, where $\tilde{h}^+_\mathrm{a}(\lambda)$ is bounded and analytic in $\lambda\in U_\delta(0)$.

In summary, for each $\lambda\in U_\delta(0)\setminus\{0\}$, every $\widetilde V\in E_-^\mathrm{u}(\lambda)$ can be written uniquely as
\begin{eqnarray*}
\widetilde V & = &
\underbrace{V^\mathrm{u}+V^\mathrm{ad}+\lambda g^\mathrm{u}_+(\lambda)(V^\mathrm{u}+V^\mathrm{ad})}_{\in E_+^\mathrm{u}(\lambda)} \\ &&
+ \underbrace{ \frac{1}{\lambda} M^{-1} V^\mathrm{ad} + \tilde{h}^+_\mathrm{a}(\lambda) (V^\mathrm{ad}+V^\mathrm{u}) + \lambda g^\mathrm{s}_+(\lambda)\left(\frac{1}{\lambda} M^{-1} V^\mathrm{ad} + \tilde{h}^+_\mathrm{a}(\lambda) (V^\mathrm{ad}+V^\mathrm{u})\right)}_{\in E_+^\mathrm{s}(\lambda)} \\ & =: &
\underbrace{V^\mathrm{u}+V^\mathrm{ad}+\lambda g^\mathrm{u}_+(\lambda)(V^\mathrm{u}+V^\mathrm{ad})}_{\in E_+^\mathrm{u}(\lambda)}
+ \underbrace{\frac{1}{\lambda} M^{-1} V^\mathrm{ad} + h^+_\mathrm{a}(\lambda) (V^\mathrm{ad}+V^\mathrm{u})}_{\in E_+^\mathrm{s}(\lambda)} \\ & = &
\underbrace{V^\mathrm{u}+V^\mathrm{ad}+\lambda g^\mathrm{u}_+(\lambda)(V^\mathrm{u}+V^\mathrm{ad})}_{\in E_+^\mathrm{u}(\lambda)}
+ \underbrace{\left(\frac{1}{\lambda} M^{-1} P_0^\mathrm{ad} + h^+_\mathrm{a}(\lambda)\right) (V^\mathrm{ad}+V^\mathrm{u})}_{\in E_+^\mathrm{s}(\lambda)} \\ & = &
\underbrace{V^\mathrm{u}+V^\mathrm{ad}+\lambda g^\mathrm{u}_+(\lambda)(V^\mathrm{u}+V^\mathrm{ad})}_{=:V\in E_+^\mathrm{u}(\lambda)}
+ \underbrace{\left(\frac{1}{\lambda} M^{-1} P_0^\mathrm{ad} + h^+_\mathrm{a}(\lambda)\right) P_0^\mathrm{u,ad} \left(V^\mathrm{u}+V^\mathrm{ad}+\lambda g^\mathrm{u}_+(\lambda)(V^\mathrm{u}+V^\mathrm{ad})\right)}_{\in E_+^\mathrm{s}(\lambda)} \\ & = &
V + \underbrace{\left(\frac{1}{\lambda} M^{-1} P_0^\mathrm{ad} + h^+_\mathrm{a}(\lambda)\right) P_0^\mathrm{u,ad}}_{=:h^+(\lambda)} V, \qquad V\in E_+^\mathrm{u}(\lambda),
\end{eqnarray*}
which is of the form stated in the lemma; note that $h^+_\mathrm{a}(\lambda)$ is analytic in $\lambda$ for $\lambda\in U_\delta(0)$. This completes the proof of the lemma.
\end{proof}

We can now extend the exponential dichotomies on $\mathbb{R}$ that exist for $\Re\lambda>0$ meromorphically into an open neighborhood of $\lambda=0$: our arguments closely follow those in \cite[\S4.2]{BSZ}. First, using the projections $P^\mathrm{s}_+(x,\lambda):=\Phi^\mathrm{s}_+(x,x,\lambda)$ and $P^\mathrm{u}_-(x,\lambda):=\Phi^\mathrm{u}_-(x,x,\lambda)$, it is easy to see that the operators
\begin{eqnarray}
\tilde{P}_+^\mathrm{s}(x,\lambda) & := & P_+^\mathrm{s}(x,\lambda) - \Phi_+^\mathrm{s}(x,0,\lambda) h_+(\lambda) \Phi_+^\mathrm{u}(0,x,\lambda),\qquad x\geq0 \label{e25} \\ \nonumber
\tilde{P}_-^\mathrm{u}(x,\lambda) & := & P_-^\mathrm{u}(x,\lambda) - \Phi_-^\mathrm{u}(x,0,\lambda) h_-(\lambda) \Phi_-^\mathrm{s}(0,x,\lambda),\qquad x\leq0
\end{eqnarray}
are bounded projections that depend meromorphically on $\lambda$ for $\lambda\in U_\delta(0)$. We can then define meromorphic continuations of the exponential dichotomies on $\mathbb{R}_\pm$ by setting
\begin{eqnarray}\label{e:ed+}
\tilde{\Phi}_+^\mathrm{s}(x,y,\lambda) & := & \Phi_+^\mathrm{s}(x,y,\lambda) \tilde{P}_+^\mathrm{s}(y, \lambda), \qquad\qquad x\geq y\geq0 \\ \nonumber
\tilde{\Phi}_+^\mathrm{u}(x,y,\lambda) & := & (1-\tilde{P}_+^\mathrm{s}(x,\lambda)) \Phi_+^\mathrm{u}(x,y,\lambda), \qquad y\geq x\geq0
\end{eqnarray}
on $\mathbb{R}_+$ and 
\begin{eqnarray}\label{e:ed-}
\tilde{\Phi}_-^\mathrm{u}(x,y,\lambda) & := & \Phi_-^\mathrm{u}(x,y,\lambda) \tilde{P}_-^\mathrm{u}(y, \lambda), \qquad\qquad 0\geq y\geq x \\ \nonumber
\tilde{\Phi}_-^\mathrm{s}(x,y,\lambda) & := & (1-\tilde{P}_-^\mathrm{u}(x,\lambda)) \Phi_-^\mathrm{s}(x,y,\lambda),\qquad 0\geq x\geq y
\end{eqnarray}
on $\mathbb{R}_-$ for $\lambda\in U_\delta(0)$.

Next, we show that $\tilde{P}_+^\mathrm{s}(0,\lambda)$ and $\tilde{P}_-^\mathrm{u}(0,\lambda)$ are complementary projections. Note that, if $P:Y\to Y$ is a bounded projection, and $h:\Rg(1-P)\to\Rg(P)$ is a bounded operator, then $\tilde{P}:=P-Ph(1-P)$ is a bounded projection with $\Rg(\tilde{P})=\Rg(P)$ and $\Ns(\tilde{P})=\graph h$: indeed, we have $\Rg(1-\tilde{P})=\Rg(1-P+Ph(1-P))=\graph h$. In particular, we see that the operators $\tilde{P}_+^\mathrm{s}(0,\lambda)$ and $\tilde{P}_-^\mathrm{u}(0,\lambda)$ satisfy
\[
\Ns(\tilde{P}_+^\mathrm{s}(0,\lambda)) = E_-^\mathrm{u}(\lambda) = \Rg(\tilde{P}_-^\mathrm{u}(0,\lambda)), \quad
\Ns(\tilde{P}_-^\mathrm{u}(0,\lambda)) = E_+^\mathrm{s}(\lambda) = \Rg(\tilde{P}_+^\mathrm{s}(0,\lambda))
\]
for $\lambda\neq0$. Thus,
\begin{equation}\label{e20}
\tilde{P}_+^\mathrm{s}(0,\lambda) = 1-\tilde{P}_-^\mathrm{u}(0,\lambda)
\end{equation}
for all $\lambda\neq0$. It follows from Lemma~\ref{lem:mer-ext} that the $\frac{1}{\lambda}$ terms in the Laurent series of the left and right-hand sides of \eqref{e20} are both given by $M^{-1}P_0^\mathrm{ad}$, so that \eqref{e20} holds for all $\lambda$.

Thus, the dichotomies in (\ref{e:ed+}) and (\ref{e:ed-}) fit together at $x=y=0$ and give the desired meromorphic exponential dichotomy on $\mathbb{R}$ for $\lambda$ near zero via
\begin{equation}\label{e:mes}
\Phi^\mathrm{s}(x,y, \lambda) := \left\{\begin{array}{lcl}
\tilde{\Phi}^\mathrm{s}_+(x,y,\lambda) & & x>y\geq0 \\[1ex]
\tilde{\Phi}_+^\mathrm{s}(x,0,\lambda)\tilde{\Phi}_-^\mathrm{s}(0,y, \lambda) & & x\geq0>y \\[1ex]
\tilde{\Phi}^\mathrm{s}_-(x,y,\lambda) & & 0>x>y
\end{array}\right.
\end{equation}
for $x>y$, and the analogous expression
\begin{equation}\label{e:meu}
\Phi^\mathrm{u}(x,y, \lambda) := \left\{\begin{array}{lcl}
\tilde{\Phi}^\mathrm{u}_+(x,y,\lambda) & & 0\leq x<y \\[1ex]
\tilde{\Phi}_-^\mathrm{u}(x,0,\lambda)\widetilde{\Phi}_+^\mathrm{u}(0,y, \lambda) & & x<0\leq y \\[1ex]
\tilde{\Phi}^\mathrm{u}_-(x,y,\lambda) & & x<y<0
\end{array}\right.
\end{equation}
for $x<y$. This completes the meromorphic extension of the exponential dichotomies on $\mathbb{R}$ for $\lambda\in U_\delta(0)$.

\subsection{Expanding the meromorphic exponential dichotomies on $\mathbb{R}$ near $\lambda=0$}\label{s:ed34}

In this section, we provide expansions of the exponential dichotomies we constructed in \S\ref{s:ed3} in the form of a Laurent series centered at $\lambda=0$. The next proposition summarizes these expansions.

\begin{proposition}\label{prop-estPhi}
Let $\Phi^\mathrm{s,u}(x,y,\lambda)$ be the exponential dichotomies constructed in \eqref{e:mes} and \eqref{e:meu}, then there exists a $\delta>0$ so that we have the Laurent-series expansion
\[
\Phi^\mathrm{s}(x,y,\lambda) =
\left\{\begin{array}{lcl} \displaystyle
-\frac1\lambda \sum_{j=1}^2 \rme^{\nu_+(\lambda) x} V_j(x) \left[ \langle W_j(y),\cdot\rangle_Y +
\lambda \rme^{-\nu_+(\lambda)y} \rmO(1) \right] & & \\[1ex] \qquad
+ \rme^{\nu_+(\lambda)(x-y)} \rmO(|\lambda| + \rme^{-\eta |y|}) + \rmO(\rme^{-\eta|x-y|}) & & x\geq y\geq 0 \\[2ex] \displaystyle
-\frac 1\lambda \sum_{j=1}^2 \rme^{\nu_+(\lambda) x} V_j(x) \langle W_j(y),\cdot\rangle_Y +
\rme^{\nu_+(\lambda)x} \rmO(\rme^{-\eta |y|}) + \rmO(\rme^{-\eta (|x-y|}) & & x\geq0>y \\[3ex] \displaystyle
-\frac 1\lambda \sum_{j=1}^2 \rme^{\nu_-(\lambda) x} V_j(x) \left[ \langle W_j(y),\cdot \rangle_Y +
\lambda \rme^{-\nu_-(\lambda)y} \rmO(1) \right] & & \\[1ex] \qquad
+ \rme^{\nu_-(\lambda)(x-y)} \rmO(|\lambda| + \rme^{-\eta |y|}) + \rmO(\rme^{-\eta|x-y|}) & & 0>x>y
\end{array}\right.
\]
for $\lambda\in U_\delta(0)$, where the $\rmO(\cdot)$ terms are operators in $L(Y)$ that depend analytically on $\lambda\in U_\delta(0)$. A similar expansion is true for $\Phi^\mathrm{u}(x,y,\lambda)$ when $x\leq y$.
\end{proposition}

\begin{remark}
We will use the expansions in Proposition~\ref{prop-estPhi} below to derive pointwise spatio-temporal bounds on the Green's function. In this analysis, it will be crucial that the terms in the Laurent series that are constant in $\lambda$, which will correspond to Gaussian behavior in the Green's function, are multiplied by either $V_j(x)$ or $\rme^{-\eta|y|}$ as this will show that the corresponding terms in the Green's functions are either exponentially localized in space or behave as derivatives of moving Gaussians.
\end{remark}

\begin{proof}
We will prove the proposition for $\Phi^\mathrm{s}(x,y,\lambda)$ as the proof for $\Phi^\mathrm{u}(x,y,\lambda)$ is very similar. Pick $\tilde\eta$ so that $0<\tilde\eta<\eta$.

First, we provide two expansions that we will use frequently in the subsequent analysis. Lemma~\ref{lem-Phipm} implies that
\[
\Phi_+^\mathrm{s}(x,0,\lambda) V_j(0) = a_j(\lambda) \rme^{\nu_+(\lambda)x} V_+^c(x,\lambda) + V^\mathrm{ss}_j(x,\lambda),\qquad x\geq0
\]
for analytic functions $a_j(\lambda)$ and $V^\mathrm{ss}_j(x,\lambda)$ with $|V^\mathrm{ss}_j(x,\lambda)|\leq C\rme^{-\eta |x|}$, where
\[
V_j(x) = a_j(0) V_+^c(x,0) + V^\mathrm{ss}_j(x,0).
\]
Hence, we have
\begin{eqnarray}
\Phi_+^\mathrm{s}(x,0,\lambda) V_j(0) & = &
a_j(\lambda) \rme^{\nu_+(\lambda)x} V_+^c(x,\lambda) + V^\mathrm{ss}_j(x,\lambda) \nonumber \\ \nonumber & = &
\rme^{\nu_+(\lambda)x} \left( a_j(\lambda) V_+^c(x,\lambda) + V^\mathrm{ss}_j(x,\lambda) \right)
+ \lambda \underbrace{\frac{1}{\lambda} \left(1-\rme^{\nu_+(\lambda)x}\right)V^\mathrm{ss}_j(x,\lambda)}_{=\rmO(\rme^{-\tilde\eta|x|})} \\ \label{e30} & = &
\rme^{\nu_+(\lambda)x} \left( V_j(x) + \rmO(\lambda) \right) + \lambda \rmO(\rme^{-\tilde\eta|x|}), \qquad x\geq0
\end{eqnarray}
where the terms on the right-hand side are analytic in $\lambda\in U_\delta(0)$. The second expansion is
\begin{equation}\label{e31}
\Phi_+^\mathrm{u}(0,y,\lambda)^* W_j(0) = W_j(y) + \lambda \rmO(\rme^{-\eta|y|}) = \rmO(\rme^{-\eta|y|}), \qquad y\geq0
\end{equation}
which follows directly from Lemma~\ref{lem-Phipm} and \eqref{eW12ed}. We can now prove the claimed estimates for $\Phi^\mathrm{s}(x,y,\lambda)$ for $x>y$, which we do separately for the three different regimes shown in the proposition.

\paragraph{The case $\mathbf{x\geq y\geq0}$:}

The equations \eqref{e25}, \eqref{e:ed+}, and \eqref{e:mes} imply that
\begin{equation}\label{e33}
\Phi^\mathrm{s}(x,y,\lambda) = \Phi_+^\mathrm{s}(x,y,\lambda) - \Phi_+^\mathrm{s}(x,0,\lambda) h^+(\lambda) \Phi_+^\mathrm{u}(0,y,\lambda)
\end{equation}
We estimate the two terms on the right-hand side separately. First, we have
\begin{eqnarray}\label{e32}
\Phi_+^\mathrm{s}(x,y,\lambda) & = & \rme^{\nu_+^c(\lambda)(x-y)} V_+^c(x,\lambda) \langle W_+^c(y,\lambda),\cdot\rangle_Y + \rmO(\rme^{-\eta|x-y|}) \\ \nonumber & = &
\rme^{\nu_+^c(\lambda)(x-y)} (V_+^c(x,0) + \rmO(\lambda)) \rmO(1) + \rmO(\rme^{-\eta|x-y|}) \\ \nonumber & = &
\rme^{\nu_+^c(\lambda)(x-y)} \sum_{j=1}^{2} b_j V_j(x,0) \rmO(1) + \rmO(\lambda\rme^{\nu_+^c(\lambda)(x-y)}+\rme^{-\eta|x-y|})
\end{eqnarray}
for some $b_j\in\mathbb{C}$, where we used that $V_+^c(0,0)$ is a linear combination of $V_1(0)$ and $V_2(0)$ by Lemma~\ref{lem-Phipm}. Next, we consider the second term in \eqref{e33}. Lemma~\ref{lem:mer-ext} shows that
\begin{eqnarray}\label{e34}
\Phi_+^\mathrm{s}(x,0,\lambda) h^+(\lambda) \Phi_+^\mathrm{u}(0,y,\lambda) & = &
\Phi_+^\mathrm{s}(x,0,\lambda) \left(\frac{1}{\lambda} M^{-1} P_0^\mathrm{ad} + h^+_a(\lambda)P_0^\mathrm{u,ad}
\right) \Phi_+^\mathrm{u}(0,y,\lambda) \\ \nonumber & = &
\frac{1}{\lambda} \Phi_+^\mathrm{s}(x,0,\lambda) M^{-1} P_0^\mathrm{ad} \Phi_+^\mathrm{u}(0,y,\lambda)
+ \rme^{\nu_+^c(\lambda)x} \rmO(\rme^{-\eta|y|}).
\end{eqnarray}
Noting that Lemma~\ref{lem:mer-ext} and \eqref{eadj3} imply that
\[
M^{-1} P_0^\mathrm{ad} = \sum_{j=1}^{2} V_j(0) \langle W_j(0),\cdot\rangle_Y,
\]
we therefore have
\begin{eqnarray}\label{e35}
\lefteqn{\frac{1}{\lambda} \Phi_+^\mathrm{s}(x,0,\lambda) M^{-1} P_0^\mathrm{ad} \Phi_+^\mathrm{u}(0,y,\lambda) =
\frac{1}{\lambda} \sum_{j=1}^{2} \Phi_+^\mathrm{s}(x,0,\lambda) V_j(0) \langle \Phi_+^\mathrm{u}(0,y,\lambda)^* W_j(0),\cdot\rangle_Y} \\ \nonumber & \stackrel{\eqref{e30}-\eqref{e31}}{=} &
\frac{1}{\lambda} \sum_{j=1}^{2}
\left( \rme^{\nu_+(\lambda)x} (V_j(x)+\rmO(\lambda)) + \lambda \rmO(\rme^{-\tilde\eta|x|}) \right)
\left( \langle W_j(y),\cdot\rangle_Y + \rmO(\lambda\rme^{-\eta|y|}) \right) \\ \nonumber & = &
\sum_{j=1}^{2}
\left( \frac{1}{\lambda} \rme^{\nu_+(\lambda)x} V_j(x) + \rme^{\nu_+(\lambda)x} \rmO(1) +\rmO(\rme^{-\tilde\eta|x|}) \right)
\left( \langle W_j(y),\cdot\rangle_Y + \rmO(\lambda\rme^{-\eta|y|}) \right) \\ \nonumber & = &
\frac{1}{\lambda} \sum_{j=1}^{2} \rme^{\nu_+(\lambda)x} V_j(x) \langle W_j(y),\cdot\rangle_Y
+ \rme^{\nu_+(\lambda)x}\rmO(\rme^{-\eta|y|})+\rmO(\rme^{-\tilde\eta|x-y|}).
\end{eqnarray}
Substituting \eqref{e35} into \eqref{e34}, we obtain
\begin{equation}\label{e36}
\Phi_+^\mathrm{s}(x,0,\lambda) h^+(\lambda) \Phi_+^\mathrm{u}(0,y,\lambda) =
\frac{1}{\lambda} \sum_{j=1}^{2} \rme^{\nu_+(\lambda)x} V_j(x) \langle W_j(y),\cdot\rangle_Y
+ \rme^{\nu_+(\lambda)x}\rmO(\rme^{-\eta|y|}) + \rmO(\rme^{-\tilde\eta|x-y|}).
\end{equation}
Substituting this expansion and \eqref{e32} into \eqref{e33}, and replacing $\tilde\eta$ by $\eta$, we arrive at the claimed expansion
\begin{eqnarray*}
\Phi_+^\mathrm{s}(x,y,\lambda) & = &
-\frac1\lambda \sum_{j=1}^2 \rme^{\nu_+(\lambda) x} V_j(x) \left[ \langle W_j(y),\cdot\rangle_Y +
\lambda \rme^{-\nu_+(\lambda)y} \rmO(1) \right] \\ & &
+ \rme^{\nu_+(\lambda)(x-y)} \rmO(|\lambda| + \rme^{-\eta|y|}) + \rmO(\rme^{-\eta|x-y|}).
\end{eqnarray*}

\paragraph{The case $\mathbf{0>x>y}$:}

In this case, \eqref{e25}, \eqref{e:ed-}, and \eqref{e:mes} imply that
\[
\Phi^\mathrm{s}(x,y,\lambda) = \Phi_-^\mathrm{s}(x,y,\lambda) + \Phi_-^\mathrm{u}(x,0,\lambda) h^-(\lambda) \Phi_-^\mathrm{s}(0,y,\lambda).
\]
Lemma~\ref{lem-Phipm} shows that the first term satisfies $\|\Phi_-^\mathrm{s}(x,y,\lambda)\|=\rmO(\rme^{-\eta|x-y|})$. The second term can be estimated in the same way as \eqref{e36}, except that $h^-(\lambda)$ contributes an additional factor $-1$; see Lemma~\ref{lem:mer-ext}. This completes the case $0>x>y$.

\paragraph{The case $\mathbf{x\geq0>y}$:}

Equations \eqref{e25}--\eqref{e:mes} show that
\begin{eqnarray*}
\Phi^\mathrm{s}(x,y,\lambda) & = &
\tilde\Phi_+^\mathrm{s}(x,0,\lambda) \tilde\Phi_-^\mathrm{s}(0,y,\lambda) \\ & = &
\Phi_+^\mathrm{s}(x,0,\lambda)\tilde{P}_+^\mathrm{s}(0,\lambda) 
(1-\tilde{P}_-^\mathrm{u}(0,\lambda))\Phi_-^\mathrm{s}(0,y,\lambda) \\ & = &
\Phi_+^\mathrm{s}(x,0,\lambda) \tilde{P}_+^\mathrm{s}(0,\lambda) \Phi_-^\mathrm{s}(0,y,\lambda) \\ & = &
\Phi_+^\mathrm{s}(x,0,\lambda)
\left(P_+^\mathrm{s}(0,\lambda) - \Phi_+^\mathrm{s}(0,0,\lambda) h_+(\lambda) \Phi_+^\mathrm{u}(0,0,\lambda)\right)
\Phi_-^\mathrm{s}(0,y,\lambda) \\ & = &
\Phi_+^\mathrm{s}(x,0,\lambda) \Phi_-^\mathrm{s}(0,y,\lambda) 
- \Phi_+^\mathrm{s}(x,0,\lambda) h_+(\lambda) P_+^\mathrm{u}(0,\lambda) \Phi_-^\mathrm{s}(0,y,\lambda) \\ & = &
\Phi_+^\mathrm{s}(x,0,\lambda) \Phi_-^\mathrm{s}(0,y,\lambda) 
- \Phi_+^\mathrm{s}(x,0,\lambda) h_+(\lambda) (P_+^\mathrm{u}(0,0)+\rmO(\lambda)) \Phi_-^\mathrm{s}(0,y,\lambda) 
\end{eqnarray*}
with $\|\Phi_-^\mathrm{s}(0,y,\lambda)\|=\rmO(\rme^{-\eta|y|})$. Hence, we obtain
\[
\Phi^\mathrm{s}(x,y,\lambda) = -\frac{1}{\lambda} \sum_{j=1}^{2} \Phi_+^\mathrm{s}(x,0,\lambda) V_j(0) \langle\Phi_-^\mathrm{s}(0,y,\lambda)^* W_j(0),\cdot\rangle_Y + \rme^{\nu_+(\lambda)x} \rmO(\rme^{-\eta|y|}) + \rmO(\rme^{-\eta|x-y|})
\]
We can now use \eqref{e30} and the fact that
\[
\Phi_-^\mathrm{s}(0,y,\lambda)^* W_j(0) = W_j(y) + \lambda \rmO(\rme^{-\eta|y|})
\]
to conclude the claimed expansion of $\Phi^\mathrm{s}(x,y,\lambda)$ in the regime $x\geq0>y$.
\end{proof}


\section{Constructing the resolvent kernel}\label{s:resolvent}

In this section, we shall construct the resolvent kernel $\widetilde{G}_{\lambda} (x,t;y,s)$, which is $2\pi$-periodic in $t$ and solves 
\begin{equation}\label{Green-eqs}
(\lambda + \partial_t - L) \widetilde{G}_{\lambda} (x,t;y,s) =  \delta(x-y) \delta(t-s)
\end{equation}
for each fixed $(s,y)$, with $L = \partial_x^2 + c_\mathrm{d} \partial_x + f_u(\bar{u})$. Having constructed the exponential dichotomy $\Phi^{\mathrm{s,u}}(x,y,\lambda)$ of the corresponding spatial dynamical system \eqref{eqs-V}, the resolvent kernel $\widetilde{G}_{\lambda} (x,t;y,s)$ can be formally constructed via the standard variation-of-constants principle; see Lemma \ref{lem-PhiSolver}. However, the source term involving $\delta(x-y) \delta(t-s)$ does not belong to the function space $Y$ (see \eqref{spaceYY}), and we first need to take care of the delta functions. 

To proceed, let us introduce $H_{\lambda,0}(x,t)$ to be the time-periodic solution to the problem 
\begin{equation}\label{def-Hlambda0}
(1+\lambda + \partial_t - D\partial_x^2 - c_\mathrm{d}\partial_x ) H_{\lambda,0}(x,t) =  \delta(x) \delta(t)
\end{equation}
and construct the resolvent kernel $\widetilde{G}_\lambda(x,t;y,s)$ of \eqref{Green-eqs} in the following form 
\begin{equation}\label{constr-Glambda} \widetilde{G}_{\lambda}(x,t;y,s) = H_{\lambda,0}(x-y,t-s) + \widetilde G_{\lambda,1}(x,t;y,s),\end{equation}
where $\widetilde G_{\lambda,1}(x,t;y,s)$ satisfies 
\begin{equation}\label{def-resGstar}
(\lambda + \partial_t - L) \widetilde G_{\lambda,1}(x,t;y,s) =  (1+f_u(\bar{u}(x,t))) H_{\lambda,0}(x-y, t-s). 
\end{equation}

The problem \eqref{def-Hlambda0} has constant coefficients and will be solved via Fourier series; see Subsection \ref{sec-H0}, below. Certainly, the right hand side of the problem \eqref{def-resGstar} is more regular than that of \eqref{Green-eqs}, though it is still not regular enough to be in $H^\frac34(\mathbb{T})$, as required in the function space $Y$. We will then iterate the previous step by introducing another auxiliary resolvent kernel $H_{\lambda, 1}(x,t;y,s) $ that takes care of the source on the right hand side of \eqref{def-resGstar}, leaving a more regular source term in the resulting equation for the remainder; see Subsection \ref{sec-H1}.

Let us now state the main results of this section. 

\begin{proposition}[Low-frequency bounds]\label{prop-reskernel} Let $\widetilde{G}_{\lambda} (x,t;y,s)$ be the resolvent kernel of $\lambda + \partial_t - L$, which is constructed in the form \eqref{constr-Glambda}. Then, for $\lambda \in \Sigma \cup B_r(0)$, we have
$$
\begin{aligned}
\widetilde{G}_{\lambda,1} (x,t;y,s) &=  - \frac 1\lambda \sum_{j=1}^2 \rme^{\nu_+(\lambda) x} P_1 V_j(x) \Big[ \langle \tilde\psi_j(y,s),\cdot \rangle 
+\lambda \rme^{- \nu_+(\lambda)y } \rmO(1)\Big]  \\&\quad + \rme^{\nu_+(\lambda)(x-y)} \rmO(|\lambda| + \rme^{-\eta |y|})  + \rmO(\rme^{-\eta|x-y|})
\end{aligned}$$
for $0\le y \le x$, 
$$
\begin{aligned}
\widetilde{G}_{\lambda,1}(x,t;y,s) &=  \frac 1\lambda \sum_{j=1}^2 \rme^{\nu_-(\lambda) x} P_1 V_j(x) \Big[ \langle \tilde\psi_j(y,s),\cdot \rangle
+\lambda \rme^{- \nu_-(\lambda)y } \rmO(1)\Big]  \\&\quad  + \rme^{\nu_-(\lambda)(x-y)} \rmO(|\lambda| + \rme^{-\eta |y|}) + \rmO(\rme^{-\eta|x-y|})
\end{aligned}$$
for $y \le x\le 0$, and 
$$
\begin{aligned}
\widetilde{G}_{\lambda,1}(x,t;y,s) &= - \frac 1\lambda \sum_{j=1}^2 \rme^{\nu_+(\lambda) x} P_1 V_j(x)  \langle \tilde\psi_j(y,s),\cdot \rangle  
+ \rme^{\nu_+(\lambda)x} \rmO(\rme^{-\eta |y|})  + \rmO(\rme^{-\eta (|x-y|})
\end{aligned}$$
for $y\le 0 \le x$, where $\tilde\psi_j(y,t)$ are functions satisfying $\tilde\psi(y,s)=\rmO(\rme^{-\eta |y|})$, $P_1V_1(x) = \bar{u}_x(x,t)$ and $P_1V_2(x) = \bar{u}_t(x,t)$. Here $\rmO(\cdot)$ are estimated in $L^\infty(\mathbb{R}\times \mathbb{T})$. Similar bounds hold 
for
$x\le y$. 
\end{proposition}

\begin{proposition}[High-frequency bounds]\label{prop-HFreskernel} Let $\widetilde{G}_{\lambda} (x,t;y,s)$ be the resolvent kernel of $\lambda + \partial_t - L$, which is constructed in the form \eqref{constr-Glambda}. Then, for $\lambda$ 
such 
that $|\Im\lambda|\le \frac 12$ and $\Re\lambda \gg1$, there holds
$$\begin{aligned}
|\widetilde  G_{\lambda,1}(x;y,s)|
 \le C_\varepsilon(1+\Re \lambda )^{-\frac12 + \varepsilon} \rme^{-\frac 12\sqrt{1+\Re \lambda }  |x-y|}
  \end{aligned}
 $$
 for arbitrary $x,y\in \mathbb{R}$, 
$\varepsilon \in (0,\frac 12]$, and some finite constant $C_\varepsilon$ which might blow up as $\varepsilon \to 0$.  
\end{proposition}

\subsection{Estimates on $H_{\lambda,0}$}\label{sec-H0}

Let us first study the problem \eqref{def-Hlambda0}. We prove the following. 

\begin{lemma}\label{lem-H0} The time-periodic solution to \eqref{def-Hlambda0} is of the form 
$$ H_{\lambda,0}(x,t) = \sum_{k\in \mathbb{Z}} \rme^{ik t} \widehat H^k_{\lambda,0}(x)$$
for $(x,t)\in\mathbb{R} \times \mathbb{T}$, where the Fourier coefficients $\widehat H^k_{\lambda,0}(x)$ satisfy 
\begin{equation}\label{bound-H0}  | \widehat H^k_{\lambda,0}(x)| \le C_0 (1+|k|)^{-\frac12} \rme^{-\frac{1}{\sqrt{D_*}}\sqrt{1+|k|}  |x|},\qquad \| \rme^{|x|}  \widehat H^k_{\lambda,0} \|_{L^p(\mathbb{R})} \le C_0(1+|k|)^{-\frac12 (1+\frac1p)}, 
\end{equation}
for $k\in \mathbb{Z}$, $p\ge1$, and for any complex $\lambda$ so that $\Re \lambda \ge -\frac12$ and $|\Im\lambda|\le \frac 12$, where where $D_* = \max\{D_1, \dots, D_n\}$ for $D = \mathrm{diag}(D_1, \dots, D_n)$. 
\end{lemma}

\begin{proof} Taking the Fourier transform of \eqref{def-Hlambda0} with respect to time, we get 
\begin{equation}\label{coeff-H00}
(1+\lambda + ik - D\partial_x^2 - c_\mathrm{d}\partial_x ) \widehat H^k_{\lambda,0}(x) =  \delta(x) .\end{equation}
That is, $\widehat H^k_{\lambda,0}(x) $ is simply the Green's kernel of the elliptic operator $1+\lambda + ik - D\partial_x^2 - c_\mathrm{d}\partial_x $, which is diagonal, $\widehat H^k_{\lambda,0}(x) = \mathrm{diag}(\widehat H^{k,1}_{\lambda,0}(x), \dots, \widehat H^{k,n}_{\lambda,0}(x))$, with components given explicitly given by  
\begin{equation}\label{def-H0l} \widehat H^{k,j}_{\lambda,0}(x) = \frac{1}{\sqrt{c_\mathrm{d}^2 + 4 D_j(1+\lambda + ik)}}\left\{ \begin{aligned}
\rme^{\mu_-(\lambda,k) x} ,&\qquad x\ge 0,
\\
\rme^{\mu_+(\lambda,k) x} ,&\qquad x \le 0,
\end{aligned}\right.\end{equation}
for $j = 1, \dots, n$, in which 
\begin{equation}\label{def-mulk}\mu^j_\pm(\lambda,k): = 
\frac{-c_\mathrm{d} \pm \sqrt{c_\mathrm{d}^2 + 4D_j (1+\lambda + ik)}}{2D_j}\end{equation}
and $\{D_j\}$ are the positive diagonal entries of the diffusion matrix: $D = \mathrm{diag}(D_1, \dots, D_n)$.
Observe that for $\Re \lambda \ge -\frac12$ and $|\Im\lambda|\le \frac 12$, 
$$\pm \Re \mu_\pm^j (\lambda , k) \ge \frac{1}{\sqrt{D_j}}\sqrt{1+|k|}$$
for all $j=1, \dots, n$. This and \eqref{def-H0l} prove at once the pointwise bound in \eqref{bound-H0} and hence the $L^p$ estimate. 
\end{proof}

\subsection{Estimates on $H_{\lambda,1}$}\label{sec-H1}
In view of Lemma \ref{bound-H0}, $H_{\lambda,0}(x,\cdot)$ belongs to $H^s(\mathbb{T})$ only for negative number $s$, and thus the right hand side of the resolvent problem \eqref{def-resGstar} is not regular enough to be in $H^{\frac 34}(\mathbb{T})$ as required in the function space $Y$.  For this reason, let us introduce an auxiliary resolvent kernel $H_{\lambda, 1}(x,t;y,s)$, which solves 
\begin{equation}\label{def-Hlambda}\left \{
\begin{aligned}
(1+\lambda + \partial_t - \partial_x^2 - c_\mathrm{d}\partial_x ) H_{\lambda, 1} (x,t;y,s)  &= (1+f_u(\bar{u}(x,t))) H_{\lambda,0}(x-y,t-s), 
\\ H_{\lambda,1}(x,s;y,s) &= 0.
\end{aligned}\right.
\end{equation}
We prove the following. 

\begin{lemma}\label{lem-Hj} For each $(y,s)$, let $H_{\lambda,1}(x,t;y,s)$ solve \eqref{def-Hlambda}. If $f_u(\bar{u})\in L^\infty (\mathbb{R}; H^1(\mathbb{T}))$, then 
\[
H_{\lambda,1} (x,\cdot; y,s)  = \rmO(\rme^{-\frac12|x-y|}),
\]
with $\rmO(\cdot)$ being estimated in $H^{1-\varepsilon}(\mathbb{T})$, for arbitrary $\varepsilon>0$. In particular, the estimate holds in $L^\infty(\mathbb{T})$, by 
Sobolev embedding. 
\end{lemma}

\begin{proof}
We again construct $H_{\lambda,1}(x,t;y,s)$ via the Fourier series. We write 
$$ H_{\lambda,1}(x,t;y,s) = \sum_{k\in \mathbb{Z}} \rme^{ik (t-s)} \widehat H^k_{\lambda,1}(x;y,s),\qquad f_u(\bar{u}(x,t)) = \sum_{k\in \mathbb{Z}} \rme^{ik t} a_k(x).$$
Note that $\widehat H^k_{\lambda,1}(x;y,s)$ are $n\times n$ complex matrices. It follows that $\widehat H^k_{\lambda,1}(x;y,s)$ solves
\[
(1+\lambda + ik - \partial_x^2 - c_\mathrm{d}\partial_x ) \widehat H^k_{\lambda, 1} (x;y,s)  =  \widehat H^{k}_{\lambda,0}(x-y) + \sum_{\ell\in \mathbb{Z}} \rme^{i\ell s} a_\ell(x) \widehat H^{k-\ell}_{\lambda,0}(x-y). 
\]   
Recall from \eqref{coeff-H00} that $ \widehat H^{k}_{\lambda,0}(x)$ is the Green's kernel of the elliptic operator $1+\lambda + ik - \partial_x^2 - c_\mathrm{d}\partial_x $. Thus, the standard variation-of-constants principle yields 
\[
\widehat H^k_{\lambda, 1} (x;y,s) = \int_\mathbb{R} \widehat H^{k}_{\lambda,0}(x-z) \Big[ \widehat H^{k}_{\lambda,0}(z-y) + \sum_{\ell\in \mathbb{Z}} \rme^{i\ell s} a_\ell(z) \widehat H^{k-\ell}_{\lambda,0}(z-y)\Big] \; dz.
\]
Using the pointwise bound \eqref{bound-H0} on $\widehat H^{k}_{\lambda,0}(x)$, we estimate   
$$\begin{aligned}
| \widehat H^k_{\lambda, 1} (x;y,s)| 
&\le C(1+ |k|)^{-\frac12}   \int_{\mathbb{R}}\rme^{-\sqrt{1+|k|} |x-z|} \Big[ (1+ |k|)^{-\frac12} \rme^{-\sqrt{1+|k|} |z-y|}\\&\qquad +\sum_{\ell \in \mathbb{Z} } \| a_\ell \|_{L^\infty} (1+|k-\ell |)^{-\frac12}\rme^{-\sqrt{1+|k-\ell|} |z-y|}  \Big] \; dz.
  \end{aligned}
 $$
Using the triangle inequality $|x-y|\le |x-z|+|z-y|$ and the fact that $\| e^{-\frac12\sqrt{1+|k|} |\cdot |}\|_{L^1} \le C (1+|k|)^{-\frac12}$, we obtain  
$$\begin{aligned}
| \widehat H^k_{\lambda, 1} (x;y,s)| 
&\le C(1+ |k|)^{-1} \rme^{-\frac12 |x-y|}  \Big[ (1+ |k|)^{-\frac12} +\sum_{\ell \in \mathbb{Z} } \| a_\ell \|_{L^\infty} (1+|k-\ell |)^{-\frac12}  \Big]
\\& \le C (1+|k|)^{-\frac32}  \rme^{-\frac12 |x-y|}  \Big[1+ \sum_{\ell \in \mathbb{Z} }  \| a_\ell \|_{L^\infty} (1+|k-\ell |)^{-\frac12} (1+|k |)^{\frac12} \Big] 
,  \end{aligned}
 $$
which is bounded by $C (1+|k|)^{-\frac32}  \rme^{-\frac12 |x-y|}  $, upon using the fact that $|k|\le (1+|\ell|)(1+|k-\ell|)$ for all $k,\ell \in \mathbb{Z}$ and the assumption $f_u(\bar{u})\in L^\infty (\mathbb{R}; H^1(\mathbb{T}))$, or equivalently the series 
$\sum_{\ell \in \mathbb{Z} } (1+|\ell |)^{\frac12} \| a_\ell \|_{L^\infty}$
is finite. 
\end{proof}

\subsection{Low-frequency bounds}

We are now ready to study the problem \eqref{def-resGstar} for the residual kernel $\widetilde G_{\lambda,1}(x,t;y,s) $ for small and bounded $\lambda$, and give a proof of Proposition \ref{prop-reskernel}. We construct $\widetilde G_{\lambda,1}(x,t;y,s) $ in the form 
\[
\widetilde{G}_{\lambda,1}(x,t;y,s) = H_{\lambda, 1}(x,t;y,s) + G_{\lambda,1}(x,t;y,s),
\]
where $H_{\lambda,1}(x,t;y,s)$ is constructed as in Lemma \ref{lem-Hj} and $G_{\lambda,1}(x,t;y,s)$ satisfies 
\begin{equation}\label{def-resGstar1}
\left\{ 
\begin{aligned}
(\lambda + \partial_t - L) G_{\lambda,1}(x,t;y,s) &=  (1+f_u(\bar{u}(x,t))) H_{\lambda,1}(x,t;y,s) ,
\\
G_{\lambda,1}(x,s;y,s) &= 0.
\end{aligned}\right. \end{equation}
Lemma \ref{lem-Hj} yields that $|H_{\lambda, 1}(x,t;y,s)| \le C e^{-\frac12 |x-y|}$, which contributes into the last term in the estimates for $\widetilde{G}_{\lambda,1}(x,t;y,s)$ as claimed in Proposition \ref{prop-reskernel}. It remains to estimate $G_{\lambda,1}(x,t;y,s)$. 

We first write the equation \eqref{def-resGstar1} in the form of the spatial dynamical system \eqref{eqs-V} with a source term
$$F(z,t;y,s) = \begin{pmatrix} 0\\ -(1+f_u(\bar{u}(z,t))) H_{\lambda,1}(z,t;y,s)\end{pmatrix} 
$$
Since $H_{\lambda,1}(z,\cdot;y,s) \in H^{\frac34}(\mathbb{T}) \cap L^\infty (\mathbb{T})$ by Lemma~\ref{lem-Hj} and $f_u(\bar{u})\in L^\infty (\mathbb{R}; H^1(\mathbb{T}))$ by assumption, the source term is in the function space $Y$; see \eqref{spaceYY}. In particular, 
there holds 
$$\|F(z,\cdot;y,s)\|_{Y} \le C \rme^{-\eta|z-y|}.$$ 
Thus, we can apply Lemma~\ref{lem-PhiSolver}, yielding the variation-of-constants formula 
\begin{equation}\label{def-Glambda1}\begin{aligned}
G_{\lambda,1}(x,t;y,s) &= P_1\int_{-\infty}^x \Phi^\mathrm{s}(x,z,\lambda) F(z,t;y,s) \; dz   - P_1\int_x^\infty \Phi^\mathrm{u} (x,z,\lambda) F(z,t;y,s) \; dz
\end{aligned}\end{equation}
where $P_1$ denotes the projection onto the first $n$ components. Using this representation, we obtain the following lemma, which completes the proof of Proposition \ref{prop-reskernel}.

\begin{lemma}\label{lem-G1} The resolvent kernel $G_{\lambda,1}(x,t;y,s)$ satisfies the same estimates as those
asserted
on 
$\widetilde{G}_{\lambda,1}(x,t;y,s) 
$
in Proposition \ref{prop-reskernel}. 
\end{lemma}
\begin{proof}
Recall that $G_{\lambda,1}(x,t;y,s)$ satisfies \eqref{def-Glambda1}, which we now estimate, following term by term in the estimates of $\Phi^\mathrm{s,u}(x,z,\lambda)$ obtained from Proposition~\ref{prop-estPhi}. We consider the case when $x\ge 0$; the other case is similar. Recall that 
$$
\begin{aligned}
\Phi^\mathrm{s}(x,y,\lambda) &=  - \frac 1\lambda \sum_{j=1}^2 \rme^{\nu_+(\lambda) x} V_j(x) \Big[ \langle W_j(y) , \cdot \rangle_Y 
+\lambda \rme^{- \nu_+(\lambda)y } \rmO(1)\Big]  \\&\quad  + \rme^{\nu_+(\lambda)(x-y)} \rmO(|\lambda| + \rme^{-\eta |y|})  + \rmO(\rme^{-\eta|x-y|}) 
\end{aligned}$$
for $0\le y \le x$, and 
$$
\begin{aligned}\Phi^\mathrm{s}(x,y,\lambda) &=  - \frac 1\lambda \sum_{j=1}^2 \rme^{\nu_+(\lambda) x} V_j(x)  \langle W_j(y) , \cdot \rangle_Y  + \rme^{\nu_+(\lambda)x} \rmO(\rme^{-\eta |y|})  + \rmO(\rme^{-\eta (|x-y|})
\end{aligned}$$
for $y\le 0\le x$. Similar bounds hold for $\Phi^\mathrm{u}(x,y,\lambda) $.

~\\
{\bf The $\frac 1\lambda $ term.} Combining the $\frac 1\lambda$ terms in $\Phi^\mathrm{s}(x,z,\lambda)$ and $\Phi^\mathrm{u}(x,z,\lambda)$, we get
$$
\begin{aligned}
 - \frac 1\lambda &\sum_{j=1}^2 \rme^{\nu_+(\lambda) x} P_1V_j(x) \int_{-\infty}^\infty \langle W_j(z) , F(z,t;y,s) \rangle_Y \; dz ,
\end{aligned}
$$
in which the integral 
may be
estimated by 
$$ \int_{-\infty}^\infty   \langle W_j(z) , F(z,t;y,s) \rangle_Y \; dz = \int_{-\infty}^\infty \rmO(\rme^{-\eta|z|} \rme^{-\eta|z-y|} ) \; dz = \rmO(\rme^{-\eta |y|}). $$

~\\
{\bf The $ \rme^{\nu_+(\lambda) x}$ term.} Let us consider the terms arriving from $\Phi^\mathrm{s}(x,z,\lambda)$. When $0\le z\le x$, these are 
$$\begin{aligned}
 - &\int_{0}^x \rme^{\nu_+(\lambda) (x-z)}  \Big[ \sum_{j=1}^2V_j(x) \rmO(1) + \rmO(|\lambda| + \rme^{-\eta |z|})\Big] F(z,t;y,s)\; dz 
 \\& = -  \rme^{\nu_+(\lambda) (x-y)}  \int_{0}^x \rme^{\nu_+(\lambda) (y-z)} \rme^{-\eta|z-y|} \Big[ \sum_{j=1}^2V_j(x) \rmO(1) +  \rmO(|\lambda| + \rme^{-\eta |z|})\Big] \; dz 
 \\& = -  \rme^{\nu_+(\lambda) (x-y)}  \int_{0}^\infty \rme^{\nu_+(\lambda) (y-z)} \rme^{-\eta|z-y|} \Big[ \sum_{j=1}^2V_j(x) \rmO(1) +  \rmO(|\lambda| + \rme^{-\eta |z|})\Big] \; dz 
 + \rmO(\rme^{-\eta |x-y|})
 \\& = -  \rme^{\nu_+(\lambda) (x-y)} \Big[ \sum_{j=1}^2V_j(x) \rmO(1) +  \rmO(|\lambda| + \rme^{-\eta |y|})\Big]  + \rmO(\rme^{-\eta |x-y|}),
  \end{aligned}
 $$
 in which the $\rmO(1)$
terms do not depend on $x$.
Similarly, when $z\le 0$, we simply get 
$$\begin{aligned}
 - \int_{-\infty}^0\rme^{\nu_+(\lambda)x} \rmO(\rme^{-\eta |z|}) F(z,t;y,s)\; dz 
&= -  \rme^{\nu_+(\lambda)x}  \int_{-\infty}^0 \rmO(\rme^{-\eta|z|} \rme^{-\eta|z-y|}) \; dz 
\\& = -  \rme^{\nu_+(\lambda)x} \rmO(\rme^{-\eta |y|}).
  \end{aligned}
 $$

~\\
{\bf The $\rme^{-\eta|x-z|}$ term.} Clearly, the integration of these terms against $F(z,t;y,s) = \rmO(\rme^{-\eta |z-y|})$ contributes an $\rmO(\rme^{-\eta |x-y|})$ term.

Similar estimates are obtained for the integral involving $\Phi^\mathrm{u}(x,z,\lambda)$. 
This completes the proof of the lemma.
\end{proof}

\subsection{High frequency bounds}
In this section, we derive estimates on the residual kernel $\widetilde G_{\lambda,1}(x,t;y,s) $ solving \eqref{def-resGstar} for sufficiently large $\Re \lambda$, and give a proof of Proposition \ref{prop-HFreskernel}. Recall that $\widetilde G_{\lambda,1}(x,t;y,s)$ solves 
\begin{equation}\label{eqs-Gl1}
\begin{aligned}
(1+\lambda + \partial_t - \partial_x^2 - c_\mathrm{d}\partial_x ) \widetilde G_{\lambda,1} (x,t;y,s)  
&= (1+f_u(\bar{u}(x,t))) H_{\lambda,0}(x-y,t-s) 
\\&\quad + f_u (\bar{u} (x,t)) \widetilde G_{\lambda,1} (x,t;y,s), 
\end{aligned}
\end{equation}
with $ \widetilde G_{\lambda,1} (x,s;y,s) = 0$. In the Fourier variables, we write 
\begin{equation}\label{def-FGl1} \widetilde G_{\lambda,1}(x,t;y,s) = \sum_{k\in \mathbb{Z}} \rme^{ik (t-s)} \widehat G^k_{\lambda,1}(x;y,s),\qquad f_u(\bar{u}(x,t)) = \sum_{k\in \mathbb{Z}} \rme^{ik t} a_k(x).\end{equation}
The Fourier coefficients then solve 
$$
\begin{aligned} 
(1+\lambda + ik - \partial_x^2 - c_\mathrm{d}\partial_x ) \widehat G^k_{\lambda, 1} (x;y,s)  &=  \widehat H^{k}_{\lambda,0}(x-y) + \sum_{\ell\in \mathbb{Z}} \rme^{i\ell s} a_\ell(x) [\widehat H^{k-\ell}_{\lambda,0} (x-y) + \widehat G^{k-\ell}_{\lambda,1}(x;y,s)],
\end{aligned}
$$
with $ \widehat H^k_{\lambda, 0} (x) $, constructed in Lemma \ref{lem-H0}. Recall from \eqref{coeff-H00} that $ \widehat H^{k}_{\lambda,0}(x)$ is the Green's kernel of the elliptic operator $1+\lambda + ik - \partial_x^2 - c_\mathrm{d}\partial_x $. Thus, the coefficients $\widehat G^k_{\lambda, 1} (x;y,s)$ satisfy the 
variation of constants formula: 
\begin{equation}\label{def-Gkl1} \widehat G^k_{\lambda, 1} (x;y,s) = \int_\mathbb{R}  \widehat H^{k}_{\lambda,0}(x-z) \Big[  \widehat H^{k}_{\lambda,0}(z-y)+ \sum_{\ell\in \mathbb{Z}} \rme^{i\ell s} a_\ell(z)[\widehat H^{k-\ell}_{\lambda,0} (z-y)+ \widehat G^{k-\ell}_{\lambda,1}(z;y,s)]\Big] \; dz \end{equation}
for each $k \in \mathbb{Z}$. We shall solve
for
 $\widehat G^k_{\lambda,1}$ using the contraction mapping theorem. Fix a $\lambda$ with $\Re \lambda \ge 0$ and $-\frac 12 \le \Im \lambda \le \frac 12$. Denote by $\mathcal{V}$ the function space that consists of functions $\widetilde G_{\lambda,1}$ 
of the
form \eqref{def-FGl1} such that the norm 
\[
\| \widetilde G_{\lambda,1}\|_{\mathcal{V}}: = \sup_{k \in \mathbb{Z}; x, y\in \mathbb{R}} (1+|\Re \lambda| + |k|)^{3/2}\rme^{\frac 12\sqrt{1+|\Re \lambda | + |k|}  |x-y|}| \widehat G^k_{\lambda, 1} (x;y,s) |
\]
is finite. Denote by $\mathcal{T}$ the map from $\mathcal{V}$ to itself for which the Fourier coefficients of $\mathcal{T}  \widetilde G_{\lambda,1}$ 
are
defined by the right hand side of \eqref{def-Gkl1}. 

It suffices to show that $\mathcal{T}$ is well-defined and contractive.  
To proceed, let $\widetilde G_{\lambda,1}$ be in $\mathcal{V}$. We first observe that the spatial eigenvalues $\mu_\pm(\lambda, k)$, defined as in \eqref{def-mulk}, satisfy 
$$
\pm \Re \mu_\pm(\lambda,k) = 
\frac{\Re \sqrt{c_\mathrm{d}^2 + 4 (1+\lambda + ik)} \mp c_\mathrm{d}  }{2} \ge \sqrt {1+\Re \lambda + |k|} =:\beta_{\lambda,k},
$$
for $\Re \lambda \ge 0$. 
By inspection 
of \eqref{def-H0l} (see also the estimate \eqref{bound-H0}), we obtain the following pointwise bound: 
\[
| \widehat H^{k}_{\lambda,0} (x)| \le C\beta_{\lambda, k}^{-1}\rme^{-\beta_{\lambda,k} |x|},
\]
for some positive constant $C$. Thus, by definition, we have 
$$\begin{aligned}
| \mathcal{T}  \widetilde G_{\lambda,1}^k(x;y,s)| &\le  \int_\mathbb{R} |\widehat H^{k}_{\lambda,0}(x-z) |\Big[ |\widehat H^{k}_{\lambda,0}(z-y) |+\sum_{\ell \in \mathbb{Z}} |a_\ell(z) [\widehat H^{k-\ell}_{\lambda,0} (z-y) + \widehat G^{k-\ell}_{\lambda,1}(z;y,s)]| \Big] \; dz
\\
&\le C\beta_{\lambda,k}^{-2}\rme^{-\frac 12\beta_{\lambda,k}|x-y|} \int_{\mathbb{R}} \rme^{-\frac 12\beta_{\lambda,k} |z-x|}\; dz  \\& \quad + C\beta_{\lambda,k}^{-1} \rme^{-\frac 12\beta_{\lambda,k} |x-y|}  \int_{\mathbb{R}}  \rme^{-\frac 12\beta_{\lambda, k} |z-x|}\sum_{\ell \in \mathbb{Z} } \| a_\ell \|_{L^\infty}  \beta_{\lambda, k-\ell}^{-1} (1+ \beta_{\lambda, k-\ell}^{-1} ) \Big] \; dz
\\
&\le C\beta_{\lambda,k}^{-3}\rme^{-\frac 12\beta_{\lambda,k}|x-y|} + C\beta_{\lambda,k}^{-2} \rme^{-\frac 12\beta_{\lambda,k} |x-y|}  \sum_{\ell \in \mathbb{Z} } \| a_\ell \|_{L^\infty}  \beta_{\lambda, k-\ell}^{-1} (1+ \beta_{\lambda, k-\ell}^{-1} ) ,
  \end{aligned}
 $$
where the integration in $z$ yields an extra factor $1/\beta_{\lambda,k}$. Next, we note that $\beta_{\lambda,k}\ge \sqrt{1+\Re \lambda}$ and $1+\Re \lambda + |k|\le (1+|\ell|)(1+ \Re \lambda +|k-\ell|)$ for all $k,\ell \in \mathbb{Z}$. Hence, we can estimate 
$$ 
\begin{aligned}
\sum_{\ell \in \mathbb{Z} } \| a_\ell \|_{L^\infty}  \beta_{\lambda, k-\ell}^{-1} (1+ \beta_{\lambda, k-\ell}^{-1} ) 
&\le \beta_{\lambda,k}^{-1} \Big[ 1+(1+\Re \lambda)^{-\frac12}\Big] \sum_{\ell \in \mathbb{Z} } (1+|\ell |)^{\frac12} \| a_\ell \|_{L^\infty}
\\
&\le C\beta_{\lambda,k}^{-1} .
\end{aligned}
$$

Recalling that $\beta_{\lambda,k} = \sqrt {1+\Re \lambda + |k|}$, this proves that 
$$\begin{aligned}
| \mathcal{T}  \widetilde G_{\lambda,1}^k(x;y,s)| &\le C (1+|\Re \lambda| + |k|)^{-3/2}\rme^{-\frac 12\sqrt{1+|\Re \lambda | + |k|}  |x-y|}
.  \end{aligned} $$
Thus, the map $\mathcal{T}$ is well-defined. In addition, for any two elements $\widetilde G_{\lambda,1}$ and $\widetilde H_{\lambda,1}$ in $\mathcal{V}$, similar calculations as above yield at once that
$$\begin{aligned}
| [\mathcal{T}  \widetilde G_{\lambda,1}^k - \mathcal{T}  \widetilde H_{\lambda,1}^k](x;y,s)| &\le  \int_\mathbb{R} |\widehat H^{k}_{\lambda,0}(x-z) |\sum_{\ell \in \mathbb{Z}} |a_\ell(z) [\widehat G^{k-\ell}_{\lambda,1} - \widehat H^{k-\ell}_{\lambda,1}](z;y,s)| \Big] \; dz
\\&\le C\beta_{\lambda,k}^{-2} \rme^{-\frac 12\beta_{\lambda,k} |x-y|}  \sum_{\ell \in \mathbb{Z} } \| a_\ell \|_{L^\infty}  \beta_{\lambda, k-\ell}^{-2} 
\\&\le C(1+\Re \lambda)^{-\frac12}\beta_{\lambda,k}^{-3} \rme^{-\frac 12\beta_{\lambda,k} |x-y|}, \end{aligned}
 $$
using the fact that $\beta_{\lambda,k}\ge \sqrt{1+\Re \lambda}$.  This proves
 $$\| \mathcal{T}\widetilde G_{\lambda,1} - \mathcal{T}\widetilde H_{\lambda,1}\|_{\mathcal{V}} \le C(1+\Re \lambda)^{-\frac12} \|\widetilde G_{\lambda,1} -\widetilde H_{\lambda,1}\|_{\mathcal{V}} .$$
Hence, $\mathcal{T}$ is contractive from the Banach space $(\mathcal{V}, \| \cdot \|_{\mathcal{V}}$) into itself, for sufficiently large $\Re \lambda$. The existence of $\widetilde G_{\lambda,1}$ in $\mathcal{V}$ solving \eqref{eqs-Gl1} follows. 
By 
definition 
of the function space $\mathcal{V}$, the Fourier coefficients in particular satisfy  
$$\begin{aligned}
|\widehat G^k_{\lambda,1}(x;y,s)|
 \le C(1+\Re \lambda)^{-\frac12 + \varepsilon}(1+|k|)^{-1 - \varepsilon} \rme^{-\frac 12 \sqrt{1+\Re \lambda} |x-y|}
  \end{aligned}
 $$
 for arbitrary $x,y\in \mathbb{R}$, $k \in \mathbb{Z}$, and arbitrary $\varepsilon \in (0,\frac 12]$. 
Hence, by definition of $H^{\frac 12 + \frac \varepsilon 2 } (\mathbb{T})$,
$$
\begin{aligned} 
\| \widetilde G_{\lambda,1}(x,\cdot;y,s) \|_{H^{\frac 12 + \frac \varepsilon 2}(\mathbb{T})}^2 &= \sum_{k\in \mathbb{Z}} (1+|k|^2)^{\frac 12 + \frac \varepsilon 2} |\widehat G^k_{\lambda,1}(x;y,s)|^2 
\\
&\le C_\varepsilon (1+\Re \lambda)^{-1 + 2\varepsilon} \rme^{- \sqrt{1+\Re \lambda} |x-y|},
\end{aligned}$$
for each $x,y\in \mathbb{R}$; in particular, by Sobolev embedding, the bound holds in $L^\infty(\mathbb{T})$. 
This completes the proof of Proposition \ref{prop-HFreskernel}.


\section{Green's function bounds}\label{s:Green's}

The main result of this section is to 
establish the pointwise bounds on the Green's function of the linearized problem about a source defect, as stated in Theorem \ref{linthm}.
Recall that the error function $e(x,t)$ and the Gaussian profile
$\theta(x,t)$ are defined as in \eqref{def-epm} and \eqref{Gaussian-like},
with
$$
\theta(x,t) = \sum_\pm \frac{1}{(1+t)^{\frac12}} \rme^{-\frac{(x- c_\pm t)^2}{M_0(t+1)}} .
$$
 

\begin{proof}[Proof of Theorem \ref{linthm}]
By Lemma \ref{lem-introGlambda}, the Green's function $G(x,t;y,s)$ is constructed via the inverse Laplace transform formula
\begin{equation}\label{def-Greenfn}
G(x,t;y,s) = \frac{1}{2\pi\rmi} \int_{\mu-\frac\rmi2}^{\mu + \frac\rmi2} \rme^{\lambda t} \widetilde{G}_{\lambda}(x,t;y,s) \; d\lambda
\end{equation} for an arbitrary constant $\mu>0$, 
in which $\widetilde{G}_{\lambda}(x,t;y,s)$ is the resolvent kernel of \eqref{Green-eqs}. Based on Propositions \ref{prop-reskernel} and \ref{prop-HFreskernel}, we write the resolvent kernel
\begin{equation}\label{res-decompose}\widetilde{G}_{\lambda}(x,t;y,s) = H_{\lambda, 0} (x-y, t-s)  + \widetilde G_{\lambda, 1} (x,t;y,s),\end{equation}
 in which $ H_{\lambda, 0} (x,t)$ is the resolvent kernel for the heat operator $1+\lambda+\partial_t - \partial_x^2 - c_\mathrm{d}\partial_x$. It is straightforward to check that the corresponding temporal Green's function of $1+\partial_t - \partial_x^2 - c_\mathrm{d}\partial_x$ is given explicitly by 
$$ H_0(x,t) = \frac{1}{\sqrt{4\pi t}} \rme^{-\frac{(x + c_\mathrm{d} t)^2}{4t}} \rme^{-t}.$$ 
Observe that for $x\ge 0$, we have $$ \rme^{-\frac{(x + c_\mathrm{d} t)^2}{4t} } \rme^{-t} = \rme^{-\frac{(x - c_+ t + (c_\mathrm{d} + c_+)t)^2}{4t}} \rme^{-t}\le \rme^{-\frac{(x - c_+ t)^2}{Mt} } \rme^{-\frac t2} ,$$
for sufficiently large $M$. Similar bound holds for $x\le 0$. This shows $H_0(x,t) \le C t^{-\frac12} \theta(x,t)$, and so $H_0(x-y,t-s)$ is indeed absorbed into the claimed bound for $G_R(x,t;y,s)$. 

In
view of \eqref{def-Greenfn} and \eqref{res-decompose}, it remains to estimate the following integral 
\begin{equation}\label{def-Greenfn1} G_1(x,t;y,s) := \frac{1}{2\pi\rmi} \int_{\mu-\frac\rmi2}^{\mu + \frac\rmi2} \rme^{\lambda t} \widetilde{G}_{\lambda, 1}(x,t;y,s) \; d\lambda\end{equation}
which corresponds to the remaining resolvent kernel $\widetilde{G}_{\lambda, 1}$. We follow the approach introduced in \cite{ZH}. In what follows, we shall consider the case when $0\le y\le x$; all other cases are similar. The proof uses the low-frequency and high-frequency resolvent bounds on $\widetilde{G}_{\lambda, 1}(x,t;y,s) $ obtained in the previous sections, depending on the relative location of $x, y$ and $t$, divided into two cases: large $\frac{|x-y|}{t}$ and bounded $\frac{|x-y|}{t}$.

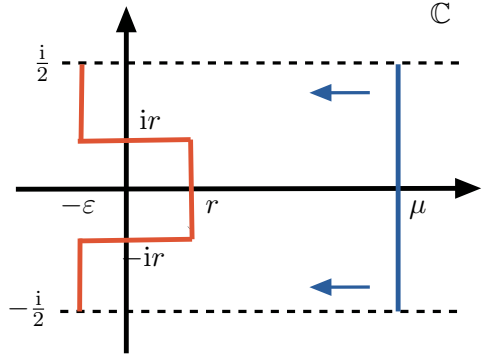
\begin{figure}
\centering
\includegraphics[scale = .7]{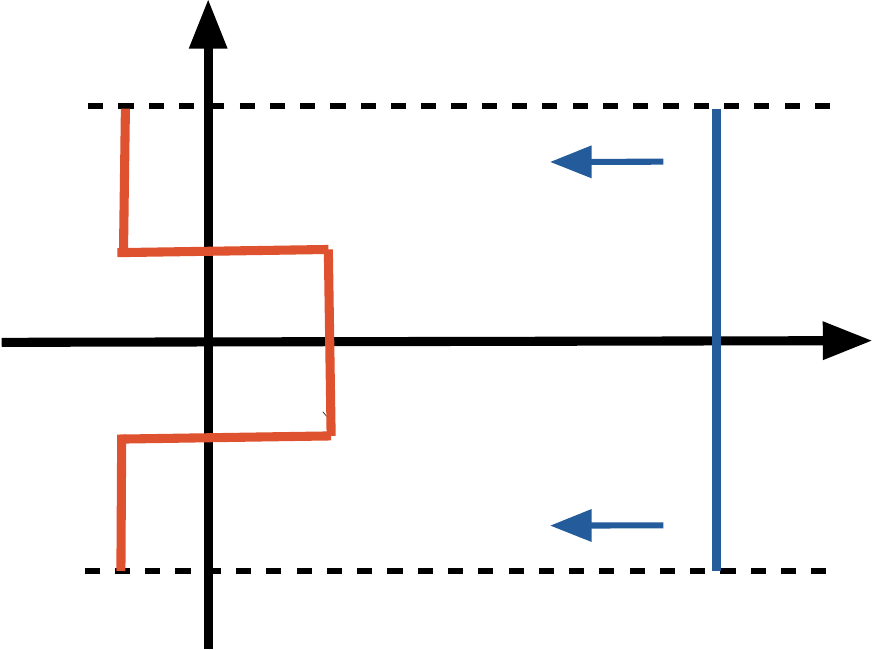}
\put(-170,107){$\frac\rmi2$}
\put(-180,14){$-\frac\rmi2$}
\put(-28,53){$\mu$}
\put(-160,53){$-\varepsilon$}
\put(-105,53){$r$}
\put(-130,85){$\rmi r$}
\put(-137,33){$-\rmi r$}
\put(-20,125){$\mathbb{C}$}
\caption{Illustrated are the initial contour of integration $\Gamma=[\mu-\frac \rmi2, \mu+\frac \rmi2]$ and the shifted contour $\Gamma_1\cup \Gamma_2$, 
in which $\Gamma_1$ consists of two segments with $\Re \lambda = -\varepsilon$ and $\Gamma_2$ denotes the three remaining segments near the origin as depicted.}
\label{fig-contour}
\end{figure}

{\bf Case (i). Large $\frac{|x-y|}{t}$.} 
In this case, we consider $x,y,t$ such that $|x-y|\ge Lt$, for some large $L$ to be determined below. 
We take $\mu$ in the integral \eqref{def-Greenfn} sufficiently large so that the estimates obtained in Proposition \ref{prop-HFreskernel} hold: namely, 
\begin{equation}\begin{aligned}\label{te}
|\widetilde  G_{\lambda,1}(x;y,s)|
 \le C_\varepsilon(1+\mu )^{-\frac12 + \varepsilon} \rme^{-\frac 12\sqrt{1+\mu }  |x-y|}
\end{aligned}
\end{equation}
 for arbitrary $x,y\in \mathbb{R}$, with $\Re \lambda = \mu$. In what follows, we take $\varepsilon =\frac 12$. 
 By a view of \eqref{def-Greenfn1}, we estimate 
\[
\begin{aligned}
 |G_1(x,t;y,s)| 
 &\le C\int_{-\frac 12}^{\frac 12} \rme^{\mu t} \rme^{-\frac 12\sqrt{1+\mu }  |x-y|}  \; d \Im \lambda 
 \le C\rme^{\mu t -\frac 12\sqrt{1+\mu }  |x-y|}.
 \end{aligned}
\]
Choosing $ \sqrt{\mu}  = \frac 14  \frac{|x-y|}{t}$, and noting that by assumption $\mu$ is therefore $\sim L^2$,
hence for $L$ sufficiently large 
here holds resolvent hold \eqref{te}, we obtain
 $|G_1(x,t;y,s)|  \le C   \rme^{-\frac{|x-y|^2}{16t}}.$
The terms $E_j$ evidently obey a similar bound, while $H(x,t;y,s)\leq C(t-s)^{-1} \rme^{-\frac{|x-y|^2}{4(t-s)}}.$
 Combining, we find that
 $$
 G_R= H+G_1 +  \bar{u}_x E_1 + \bar{u}_t E_2
 $$
decays as $C(t-s)^{-1} \rme^{-\frac{|x-y|^2}{16(t-s)}}.$
which is readily seen to be absorbable in the claimed bounds.

{\bf Case (ii). Bounded $\frac{|x-y|}{t}$.}
We now consider the case when $|x-y|\le Lt$, for the large, but fixed, constant $L$ chosen above.
In this case, applying the Cauchy's integral theorem, we move the contour of integration $\Gamma$ into $\Gamma_1 \cup \Gamma_2$ as depicted in Figure \ref{fig-contour}. Here, we take $r$ and $\varepsilon$ sufficiently small so that Proposition \ref{prop-reskernel} holds for $\lambda\in \Gamma_2$.  On $\Gamma_1 = [-\varepsilon + \rmi r, -\varepsilon + \frac \rmi 2] \cup [-\varepsilon -\frac \rmi 2, -\varepsilon - \rmi r]$, the resolvent kernel $\widetilde G_\lambda$ and hence $\widetilde{G}_{\lambda,1} $ are holomorphic, and in particular bounded. Hence, we can estimate 
$$
\begin{aligned}
\Big| \frac{1}{2\pi\rmi} \int_{\Gamma_1} \rme^{\lambda t} \widetilde{G}_{\lambda, 1}(x,t;y,s) \; d\lambda \Big| 
&\le C \rme^{-\varepsilon t}
\le C \rme^{-\frac \varepsilon2 t} \rme^{- \frac{\varepsilon}{2L}|x-y|}
\end{aligned}$$  
which is bounded by $C\rme^{-\eta (|x-y|+t)}$ for some small $\eta$,
hence, for bounded $|x-y|/t$, absorbable into the claimed bounds. 

Next, we estimate the integral on $\Gamma_2$. For $\lambda \in \Gamma_2$, since $r$ is sufficiently small, we can use the representation obtained in Proposition \ref{prop-reskernel}, giving   
\begin{eqnarray*}
\widetilde{G}_{\lambda,1} (x,t;y,s) & = & - \frac 1\lambda \sum_{j=1}^2 \rme^{\nu_+(\lambda) x} P_1 V_j(x) \Big[ \langle \tilde\psi_j(y,s) , \cdot \rangle
+\lambda \rme^{- \nu_+(\lambda)y } \rmO(1)\Big] \\ &&
+ \rme^{\nu_+(\lambda)(x-y)} \rmO(|\lambda| + \rme^{-\eta |y|})  + \rmO(\rme^{-\eta|x-y|}).
\end{eqnarray*}
Such an integral for small $\lambda$ has appeared and been estimated in the literature; see, for instance, the original contribution \cite[Section 8]{ZH} or its applications in other contexts such as \cite{MZ1,BNSZ2}, to name a few. See also \cite[Section 5.1]{BSZ} which deals with the time-periodic setting. As a result, the integral over $\Gamma_2$ of $\frac 1\lambda \rme^{\nu_+(\lambda) x} $, with $\nu_+(\lambda) = -\lambda /c_+ + \rmO(\lambda^2)$, contributes the error function $e(x,t-s)$, and that of $ \rme^{\nu_+(\lambda)(x-y)} $ yields the Gaussian-like behavior $\theta(x-y, t-s)$. We note that both of these integrals are multiplied either by $\bar{u}_x$ or $\bar{u}_t$. The remaining terms are estimated by moving contours appropriately, as in \cite{MZ1,ZH}, and noting that an extra factor of $|\lambda|$ yields an extra time decay of order $(t-s)^{-\frac12}$. For further details we refer the reader to one of the mentioned references.  
\end{proof}


\begin{remark}
In \cite{BSZ}, high-frequency resolvent bounds were not needed, as
$\rme^{-\eta(t+|x-y|)}$ decay in the far field $|x-y|\gg t$ was sufficient for the arguments used there. 
Here, because we are dealing with data decaying at Gaussian
rate, we require a Gaussian decay in our Green's function estimates in order to show that this tail decay is
propagated in time.
\end{remark}

\section{Nonlinear stability}\label{s:stab}

Let $\widetilde u(x,t)$ be the solution to the nonlinear equation \eqref{eqs-RD}, and let $\bar{u}(x,t)$ be the source defect solution. The invariance under space and time translation of the nonlinear equation implies that  $(\partial_t - L) \bar{u}_x = (\partial_t - L)\bar{u}_t = 0$, where $L$ is the linearized operator about $\bar{u}$. We are led to introduce the ansatz \begin{equation}\label{def-v}
v(x,t) := \widetilde{u}(x+\psi(x,t),t+\varphi(x,t))-\bar{u}(x,t)
\end{equation}
which is formally approximated by $\psi \bar{u}_x + \varphi \bar{u}_t$, plus the residual term. Here, the phase shifts $\varphi$ and $\psi$ are to be determined. 

\begin{lemma}\label{cancel}
	The perturbation $v$ defined in \eqref{def-v} satisfies
\begin{equation}\label{eqs-pertubation} (\partial_t - L) (v  - \bar{u}_x \psi - \bar{u}_t \varphi)  = Q(v,\varphi,\psi),\end{equation}
with
\begin{equation}\label{quadratic-bound}
 Q(v,\varphi, \psi) \quad \lesssim \quad 
|v|^2 + \sum_{k=0}^1 \Big[ |D^{k+1}_{x,t}\varphi| +|D^{k+1}_{x,t}\psi|\Big]^2
+ |\mathcal{D}_{x,t}v| \sum_{k=0}^1 \Big[ |D^{k+1}_{x,t}\varphi| +|D^{k+1}_{x,t}\psi| \Big]
 \end{equation}
as long as $v,\mathcal{D}_{x,t}\psi,\mathcal{D}_{x,t}\varphi$ and their derivatives remain bounded and small. 
\end{lemma}

\begin{proof}
	In what follows $Q_j(v,\varphi, \psi)$ denotes different functions satisfying the quadratic bound \eqref{quadratic-bound}. For convenience, we set $\xi: = x + \psi(x,t)$ and $\tau: = t + \varphi(x,t)$. The nonlinear equation \eqref{eqs-RD} now reads 
\begin{equation}\label{eqs-prtb-ugly}
(\partial_\tau  - c_\mathrm{d}\partial_\xi - \partial_\xi^2) (\bar{u}(x,t) + v(x,t)) = f(\bar{u} (x,t)+ v(x,t)).\end{equation} 
To compute the left-hand side, we note that by definition 
$$ \begin{pmatrix} \partial_\tau \\ \partial_\xi \end{pmatrix}  = \frac{1}{d(x,t)} 
\begin{pmatrix} 1+\psi_x & - \psi_t \\ - \varphi_x & 1 + \varphi_t \end{pmatrix}
 \begin{pmatrix} \partial_t \\ \partial_x\end{pmatrix} ,\qquad d(x,t): = (1+\varphi_t )(1+\psi_x) - \varphi_x \psi_t. $$
Since $d =  1+\varphi_t +\psi_x + Q_1$, we have $\frac 1d = 1 - \varphi_t - \psi_x + Q_2$, for $Q_1, Q_2$ 
satisfying \eqref{quadratic-bound}. This yields 
$$ \begin{pmatrix} \partial_\tau \\ \partial_\xi \end{pmatrix}  = \Big[ I  -  \begin{pmatrix} \varphi_t  & \psi_t \\  \varphi_x & \psi_x \end{pmatrix}+ Q_3 \Big] \begin{pmatrix} \partial_t \\ \partial_x\end{pmatrix} .$$

Substituting into \eqref{eqs-prtb-ugly}, we thus obtain
$$(\partial_\tau  - c_\mathrm{d}\partial_\xi - \partial_\xi^2) v = (\partial_t  - c_\mathrm{d}\partial_x - \partial_x^2) v + Q_4(v,\varphi,\psi) $$
where
$Q_4$ satisfies \eqref{quadratic-bound}. 
Similarly, we have 
$$(\partial_\tau  - c_\mathrm{d}\partial_\xi ) \bar{u} =  (\partial_t  - c_\mathrm{d}\partial_x ) \bar{u}  -\bar{u}_t (\partial_t - c_\mathrm{d} \partial_x) \varphi  -\bar{u}_x (\partial_t - c_\mathrm{d} \partial_x) \psi    + Q_5(\varphi,\psi)  $$
and 
$$
\begin{aligned}
\partial_\xi^2 \bar{u}
&= \Big( (1- \psi_x) \partial_x - \varphi_x \partial_t\Big)^2 \bar{u} + Q_6(\varphi,\psi)
\\
&=  \partial_x^2 \bar{u} - \bar{u}_t \partial_x^2 \varphi  - \bar{u}_x \partial_x^2 \psi  - 2\varphi_x\bar{u}_{xt}  - 2\psi_x \bar{u}_{xx}  + Q_7(\varphi,\psi).
\end{aligned}$$
Putting these calculations together into \eqref{eqs-prtb-ugly}, we obtain 
\begin{equation}\label{comp-ugly} \begin{aligned}
0&= (\partial_\tau  - c_\mathrm{d}\partial_\xi - \partial_\xi^2) (\bar{u} + v) - f(\bar{u} + v) 
\\
&=  (\partial_t  - c_\mathrm{d}\partial_x - \partial_x^2) v +  (\partial_t  - c_\mathrm{d}\partial_x - \partial_x^2) \bar{u} 
 -\bar{u}_t (\partial_t - c_\mathrm{d} \partial_x - \partial_x^2) \varphi 
\\&\quad -\bar{u}_x (\partial_t - c_\mathrm{d} \partial_x - \partial_x^2) \psi   
 + 2\varphi_x\bar{u}_{xt}  + 2\psi_x \bar{u}_{xx}  - f(\bar{u} + v) + Q_8(v,\varphi,\psi).
\end{aligned}\end{equation}
Finally, 
using $(\partial_t - L) \bar{u}_x = (\partial_t - L)\bar{u}_t = 0$, with $L = \partial_x^2+ c_\mathrm{d}\partial_x+ f_u(\bar{u})$, we 
compute 
$$
\begin{aligned}
(\partial_t - L) (\bar{u}_x \psi)  &= \bar{u}_x (\partial_t - c_\mathrm{d} \partial_x - \partial_x^2) \psi   
-2\psi_x \bar{u}_{xx} ,
\\
(\partial_t - L) (\bar{u}_t \varphi)  &= \bar{u}_t (\partial_t - c_\mathrm{d} \partial_x - \partial_x^2) \varphi 
- 2\varphi_x\bar{u}_{xt} .
\end{aligned}
$$
Putting these into \eqref{comp-ugly} and using the nonlinear equation for $\bar{u}$, one has
$$\begin{aligned}
0&=  (\partial_t  - L) (v  -\bar{u}_t \varphi  -\bar{u}_x  \psi  ) - f(\bar{u} + v)  + f(\bar{u}) + f_u (\bar{u}) v  + Q_8(v,\varphi,\psi),
\end{aligned}$$
which gives at once the equation \eqref{eqs-pertubation}, with $Q(v,\varphi,\psi) = Q_8(v,\varphi, \psi) + \mathcal{O}(v^2)$. 
\end{proof}

\subsection{Integral representations}\label{s:int}
By view of Lemma \ref{cancel}, the Duhamel principle yields
\begin{eqnarray*}
v(x,t) & = & \bar{u}_x \psi (x,t)+ \bar{u}_t \varphi (x,t)+  \int G(x,t; y,0) (v - \bar{u}_x \psi - \bar{u}_t \varphi)(y,0) \; \rmd y 
\\ &&
+ \int_0^t \int G(x,t;y,s) Q(v,\varphi,\psi)(y,s)  \; \rmd y\rmd s
\end{eqnarray*}
so long as $v$, $\psi$, $\varphi$, and $Q(v,\varphi, \psi)$ remain bounded in $L^2$.  

Our goal is to show that
for an appropriate choice of $\psi$ and $\varphi$, 
$v(x,t)$ tends to zero as $t\to \infty$: more 
precisely, that $v(x,t)$ 
is bounded by the Gaussian profile
$\theta(x,t)$. Exploiting the Green's function decomposition given in Theorem \ref{linthm}, we shall choose the shifts $\varphi, \psi$ so that the non-decaying or slow-decaying terms in the Green's function are cancelled out with $\bar{u}_x \psi + \bar{u}_t \varphi$. Let us analyze each term on the right-hand side of the above integral formula. 
Set
\begin{equation}\label{def-v0}v_0(x): = \widetilde u(x,  0) - \bar{u}(x,0) , \qquad \varphi(x,0) = \psi(x,0) =0.\end{equation}
Recalling the Green's function decomposition \eqref{Gdecomp} obtained in Theorem \ref{linthm}
$$ G(x,t;y,s) =  \bar{u}_x(x,t) E_1(x,t; y,s) + \bar{u}_t(x,t) E_2(x,t; y,s) + G_R(x,t; y,s),
$$
we write
$$\begin{aligned}
 \int &G(x,t; y,0)  v_0 (y) \; \rmd y  
 \\& =  \int [ \bar{u}_x(x,t) E_1(x,t;y,0) + \bar{u}_t(x,t)E_2(x,t;y,0) + G_R(x,t;y,0)]  v_0 (y) \; \rmd y
\end{aligned}$$
and $$\begin{aligned}
\int_0^t \int &G(x,t;y,s) Q(v,\varphi,\psi)  \; \rmd y\rmd s ,
\\ & = \int_0^t \int [ \bar{u}_x(x,t) E_1(x,t;y,s) + \bar{u}_t(x,t)E_2(x,t;y,s) + G_R(x,t;y,s)] Q(v,\varphi,\psi)  \; \rmd y\rmd s.
\end{aligned}$$
Pulling out $\bar{u}_x(x,t)$ and $\bar{u}_t(x,t)$ in the above expression, we are led to introduce 
\begin{equation}\label{def-shifts} \begin{aligned}
 \psi(x,t) : &= -  \int  E_1(x,t;y,0) v_0 (y) \; \rmd y - \int_0^t \int  E_1(x,t;y,s)Q(v,\varphi,\psi) (y,s)  \rmd y\rmd s \\
 \varphi(x,t) : &= -  \int E_2(x,t;y,0) v_0(y)\; \rmd y  - \int_0^t \int  E_2(x,t;y,s) Q(v,\varphi,\psi)  (y,s ) \rmd y\rmd s.
 \end{aligned}\end{equation}
We note that by definition, $E_j(x,0;y,0) =0$ and hence the above setting is consistent with \eqref{def-v0};
indeed, $\psi, \varphi \equiv 0$ for $0\leq t\leq 1$. 
 With \eqref{def-shifts}, the integral representation for $v(x,t)$ then reduces to 
\begin{equation}\label{def-vint01} v(x,t) =  \int G_R(x,t; y,0)  v_0 (y) \; \rmd y +  \int_0^t \int G_R(x,t;y,s) Q(v,\varphi,\psi) (y,s) \; \rmd y\rmd s .\end{equation}



\subsection{Nonlinear iteration}\label{s:it}
Set $b:=(v,\mathcal{D}_{x,t}\psi,\mathcal{D}_{x,t}\varphi)$ and introduce the weighted space-time norm
\[
\|b\|_1
 := \sup_{0 \leq s \leq t, \; y\in \mathbb{R}} \theta^{-1} \sum_{k=0}^1 \Big( 
|v| + s(1+s)^{-1}|\mathcal{D}_{y,s}^kv| 
+ |D^{k+1}_{y,s}\varphi| + |D^{k+1}_{y,s}\psi|  \Big) (y,s) 
\]
where $\theta(x,t)$ denotes the Gaussian profile defined in \eqref{Gaussian-like}. We denote by $\mathcal{B}$ the associated Banach space with the norm $\|\cdot \|_1$. 
We note that the constant $M_0$ in \eqref{Gaussian-like} is a fixed, large, positive number. At various points in the below estimates, there will be a similar quantity, which we denote by $M$, that will need to be taken to be sufficiently large. The number $M_0$ is then the maximum value of $M$, at the end of the proof. 

Observe that the nonlinearity $Q(v,\varphi,\psi)$ introduced in Lemma \ref{cancel} depends only on $b$. Thus, we may define a mapping $\mathcal{T}$ on the Banach space $(\mathcal{B},\|\cdot \|_1)$
given by the righthand side of \eqref{def-vint01} and the space-time derivatives of the righthand sides of
\eqref{def-shifts}. Precisely, for each $b=(v,\mathcal{D}_{x,t}\psi,\mathcal{D}_{x,t}\varphi)$ in $\mathcal{B}$, we set 
\begin{equation}\label{def-TTT} \mathcal{T} = (\mathcal{T}_{v}b , \mathcal{T}_{\psi_{x,t}}b , \mathcal{T}_{\varphi_{x,t}}b)\end{equation}
with 
$$\begin{aligned}
\mathcal{T}_{v}b (x,t)  &=  \int G_R(x,t; y,0)  v_0 (y) \; \rmd y +  \int_0^t \int G_R(x,t;y,s) Q(v,\varphi,\psi) (y,s) \; \rmd y\rmd s
\\	\mathcal{T}_{\psi_{x,t}}b 
	(x,t)  &=  - \int \mathcal{D}_{x,t} E_1(x,t;y,0) v_0(y) \rmd y - \int_0^t \int \mathcal{D}_{x,t} E_1(x,t;y,s)Q(v,\varphi,\psi)   \rmd y\rmd s \\
	\mathcal{T}_{\varphi_{x,t}}b(x,t)  &=  - \int \mathcal{D}_{x,t} E_2(x,t;y,0) v_0(y)  \rmd y - \int_0^t \int \mathcal{D}_{x,t}E_2(x,t;y,s) Q(v,\varphi,\psi)   \rmd y\rmd s.
 \end{aligned}$$

The key issue therefore is to show that this defines a contraction mapping from 
a small ball $B(0,\varepsilon)\subset \mathcal{B}$ to itself, thus establishing at the same time existence and 
the desired time-asymptotic bounds.
This follows readily from the following main estimate.

\begin{proposition} \label{prop-iteration} Let $C_0>0$. There exist positive constants $\varepsilon_0$ and $C_1$ such that
\begin{equation}\label{keyh-ineq}
\begin{aligned}
	\|\mathcal{T}b\|_1 \;\le\; C_1 (\varepsilon  + \|b\|_1^2), 
	\qquad
	\|\mathcal{T}b -\mathcal{T}\widetilde b\|_1 \;\le\; C_1 \max\{\|b\|_1,\|\widetilde b\|_1\} \|b-\widetilde b\|_1
\end{aligned}
\end{equation}
for all initial data $\widetilde u(x,0)$ such that $\varepsilon:=\| \rme^{|\cdot|^2/C_0}(\widetilde u - \bar{u})(\cdot,0) \|_{C^2}\leq\varepsilon_0$
and for all $b,\widetilde b\in \mathcal{B}$ with $\|b\|_1,\|\widetilde b\|_1$ sufficiently small.  
\end{proposition}

With this result in hand, we obtain the main results by the Banach contraction mapping theorem.
The remainder of this section is devoted to the proof of Proposition \ref{prop-iteration}. 

In view of the estimates on the nonlinear remainder term $Q(v,\varphi,\psi)$, satisfying the quadratic bound \eqref{quadratic-bound}, and the definition of $\|\cdot\|_1$, we obtain at once 
\begin{equation} \label{non-est} 
Q(v,\varphi,\psi)  \lesssim  \|b\|_1^2 \theta(x,t) ^2. 
\end{equation}
We now use the integral representations \eqref{def-shifts} and \eqref{def-vint01} to prove the iterative estimate \eqref{keyh-ineq}. First, let us recall the following convolution estimates from \cite{HZ}.
\begin{lemma} \label{lem-conv} There is some constant $C$ such that for $t\geq s$, $j=1,2$, $k=0,1$,
\begin{subequations}
\begin{align}
	\int_0^t \int_{\mathbb{R}} |\mathcal{D}_{x,t}^k G_R(x,t;y,s) |\theta^2 (y,s)  \; \rmd y\rmd s  \quad &\leq \quad C\theta(x,t), \label{conv1}
\\
\int_0^t \int_{\mathbb{R}} |\mathcal{D}_{x,t}^{1+k} E_j(x,t;y,s)|\theta^2 (y,s)  \; \rmd y\rmd s   \quad &\leq \quad C\theta(x,t) \label{conv2}
\end{align}
\end{subequations}
for all $t\geq 0$.
(Here as before, $\mathcal{D}_{x,t}$ denotes $(\partial_x,\partial_t)$.) 
\end{lemma} 

\begin{lemma}\label{lem-inconv} There is some constant $C$ such that, for 
$j=1,2$, $k=0,1$, and all $t\geq 0$,
\begin{subequations}
\begin{align}
	\int |\mathcal{D}_{x,t}^kG_R(x,t; y,0) | \rme^{-|y|^2/C_0} \; \rmd y   \quad &\leq \quad C t^{-1}(1+t)\theta(x,t) , 
\label{inconv1}\\
	\int_{\mathbb{R}} |\mathcal{D}_{x,t}^{1+k}E_j(x,t; y,0)| \rme^{-|y|^2/C_0} \; \rmd y  \quad &\leq \quad C\theta(x,t). 
\label{inconv2}
\end{align}
\end{subequations}
\end{lemma}

The proofs of these estimates follow closely to those 
in \cite{HZ, BNSZ1, BNSZ2}. 

\begin{proof}[Proof of Lemma~\ref{lem-conv}] 


~\\
{\bf Proof of \eqref{conv1}.}
Recall, for $k=0,1$, that
\[
	\mathcal{D}_{x,t}^k G_R (x,t;y,s) \quad \lesssim\quad  (t-s)^{-k/2} ((t-s)^{-\frac12} + \rme^{-\eta |y|}) \theta(x-y,t-s),
\]
where $\theta(x,t)$ is the Gaussian profile
$
\theta(x,t) = \sum_\pm \frac{1}{(1+t)^{\frac12}} \rme^{-\frac{(x+ c_\pm t)^2}{M_0(t+1)}} 
$
(Theorem \ref{linthm}).

We wish to show that \begin{equation}\label{e:d1}
	I: = \theta(x,t)^{-1} \int_0^t \int_\mathbb{R} |\mathcal{D}_{x,t}^k G_R(x,t;y,s)| \theta^2(y,s)  \rmd y \rmd s =   I_1 + I_2
\end{equation}
is bounded for all $t\ge 0$, where $ I_1,I_2$ correspond to the terms involving $(t-s)^{-\frac12}\theta$, and $\rme^{-\eta |y|}\theta$, respectively, in $G_R(x,t;y,s)$. Combining only the exponentials in this expression, we obtain terms that can be bounded by
\begin{equation}\label{exp-form}
\exp\left(\frac{(x+\alpha_3 t)^2}{M(1+t)} - \frac{(x-y+\alpha_1 (t-s))^2}{4(t-s)} - \frac{(y+\alpha_2 s)^2}{M(1+s)} \right)
\end{equation}
with $\alpha_j=c_\pm$. To estimate this  expression, we proceed as in \cite[Lemma~7]{HZ}, and complete the square of the last two exponents in (\ref{exp-form}). Written in a slightly more general form, we obtain
\begin{eqnarray*}
\lefteqn{\frac{(x-y-\alpha_1(t-s))^2}{M_1(t-s)} + \frac{(y-\alpha_2 s)^2}{M_2(1+s)} \;=\;
\frac{(x-\alpha_1(t-s)-\alpha_2 s)^2}{M_1(t-s)+M_2(1+s)} } \\ &&
+ \frac{M_1(t-s)+M_2(1+s)}{M_1M_2(1+s)(t-s)}\left( y - \frac{xM_2(1+s) - (\alpha_1M_2(1+s) + \alpha_2M_1s)(t-s)}{M_1(t-s)+M_2(1+s)}\right)^2
\end{eqnarray*}
and conclude that the exponent in (\ref{exp-form}) is of the form
\begin{eqnarray}\label{exp-form-2}
\lefteqn{ \frac{(x+\alpha_3t)^2}{M(1+t)} - \frac{(x-\alpha_1(t-s)-\alpha_2s)^2}{4(t-s)+M(1+s)} } \\ \nonumber &&
- \frac{4(t-s)+M(1+s)}{4M(1+s)(t-s)}\left( y - \frac{xM(1+s) - (\alpha_1M(1+s) + 4\alpha_2s)(t-s)}{4(t-s)+M(1+s)}\right)^2,
\end{eqnarray}
with $\alpha_j=c_\pm $. Using that the maximum of the quadratic polynomial $\alpha x^2+\beta x+\gamma$ is $-\beta^2/(4\alpha)+\gamma$, 
we
see that the sum of the first two terms in (\ref{exp-form-2}), which involve only $x$ and not $y$, is less than or equal to zero. Omitting this term, we therefore obtain the estimate
\begin{eqnarray*}
\lefteqn{ \exp\left( \frac{(x+\delta_3 t)^2}{M(1+t)} - \frac{(x-y\delta_1(t-s))^2}{4(t-s)} - \frac{(y-\delta_2s)^2}{M(1+s)} \right) } \\ & \leq &
\exp \left( - \frac{4(t-s)+Ms}{4M(1+s)(t-s)}\left( y - \frac{xM(1+s)+(\delta_1M(1+s)+ 4\delta_2s)(t-s)}{4(t-s)+M(1+s)}\right)^2 \right)
\end{eqnarray*}
for $\delta_j=c_\pm$. Using this result, we can now estimate the first term $I_1$ in \eqref{e:d1}. Indeed, we have
\begin{eqnarray*}
	I_1 & \le & C_1 (1+t)^{\frac12}\int_0^t \frac{1}{(t-s) (1+s)} \\ && \times 
\int_\mathbb{R} \exp\left(- \frac{4(t-s)+M(1+s)}{4M(1+s)(t-s)}\left( y - \frac{[xM(1+s) + c_\pm (M(1+s) + 4s)(t-s)]}{4(t-s)+M(1+s)}\right)^2 \right) \rmd y\rmd s \\ & \leq &
	C_1(1+t)^{\frac12}\int_0^t \frac{1}{(t-s)(1+s)}  \sqrt{\frac{4M(1+s)(t-s)}{4(t-s)+M(1+s)}} \rmd s 
\\ & \leq &
C_1(1+t)^{\frac12}\int_0^{t/2} \frac{1}{ t(1+s)^{\frac12}} \rmd s + C_1(1+t)^{\frac12}\int_{t/2}^t \frac{1}{(t-s)^{\frac12} (1+t)}\rmd s \\ & \leq &
C_1.
 \end{eqnarray*}
This proves the estimate for $I_1$. The terms $I_2$ are easier using the fact that the spatial decay $\rme^{-\eta|y|}$ 
mulitplied against a Gaussian profile
yields 
exponential decay in time and space. This proves the first convolution estimate in Lemma~\ref{lem-conv}. 

~\\
{\bf Proof of \eqref{conv2}.}  Recall that 
$$E_j (x,t; y,s) =  e(x,t-s) \beta_j(y) + G_j(x,t; y,s), \qquad j=1,2,$$
where $G_j(x,t; y,s)$ behaves like a Gaussian and $\beta_j(y)$ is localized: $|\beta_j(y)|\le C \rme^{-\eta|y|}$. We observe that $\mathcal{D}_{x,t}G_j(x,t; y,s)$ enjoys the same bounds as $G_R(x,t;y,s)$ and thus, for $k=0,1$,
$$ \int_0^t \int_{\mathbb{R}} |\mathcal{D}_{x,t}^{1+k} G_j(x,t;y,s)|\theta^2 (y,s)  \; \rmd y\rmd s   \quad \leq \quad C\theta(x,t).$$
It remains to check the first term in $E_j(x,t;y,s)$. We note that 
$
\rme^{-\eta|y|} \rme^{-\frac{(y+c_\pm t)^2}{M(1+t)}} \leq C \rme^{-\frac{\eta|y|}{2}} \rme^{-\frac{c^2 t}{M}},
$
hence
 $$\begin{aligned}
  \int_0^t \int_{\mathbb{R}} & |\mathcal{D}_{x,t} e(x,t-s)\beta_j(y)| \theta^2 (y,s)  \; \rmd y\rmd s   
  \\& \le  C \int_0^t \int_{\mathbb{R}} |\mathcal{D}_{x,t} e(x,t-s)| \rme^{-\frac{\eta|y|}{2}} \rme^{-\frac{c^2 t}{M}} \; \rmd y\rmd s  
  \\& \le  C \int_0^t \int_{\mathbb{R}} |\mathcal{D}_{x,t}\Big[ e(x,t) +  \theta(x-y,t-s) \Big] |\rme^{-\frac{\eta|y|}{2}} \rme^{-\frac{c^2 t}{M}} \; \rmd y\rmd s  
  \\& \le  C |\mathcal{D}_{x,t} e(x,t)| +  \int_0^t \int_{\mathbb{R}} |\mathcal{D}_{x,t} \theta(x-y,t-s) \rme^{-\frac{\eta|y|}{2}} \rme^{-\frac{c^2 t}{M}} \; \rmd y\rmd s  
\\&\le C \theta(x,t) .   \end{aligned}$$
Here we have used the facts that $e(x,t-s) \rme^{-\eta(|y|+s)}  = e(x,t) \rme^{-\eta(|y|+s)} $ up to 
a Gaussian-order error
(for instance, see \cite[Proof of Lemma~6.6]{BNSZ2}), and that derivatives of $e(x,t)$ are also of Gaussian order.  
The last integral term involving $\mathcal{D}_{x,t}\theta(x-y,t-s)$ was already treated in the first convolution estimate; see the proof of the $I_1$ estimate in \eqref{e:d1}.  
\end{proof}

\begin{proof}[Proof of Lemma~\ref{lem-inconv}] 
In the case $k=0$, the estimate \eqref{inconv1} can be found, for instance, in the proof of Lemma 3, \cite{HZ}. 
The derivative bound $k=1$ goes similarly, noting the additional factor $t^{-1}(1+t)$ appearing in the derivative of $G_R$.
For the other estimates, we note that 
$$ \mathcal{D}_{x,t}\chi(t)E_j(x,t;y,0) = \mathcal{D}_{x,t}\chi(t)e(x,t) \beta_j(y) + \mathcal{D}_{x,t} \chi(t) G_j(x,t; y,0) .$$
This vanishes for $t\leq 1$, where $\chi\equiv 0$, hence \eqref{inconv2} follows trivially in this case.
For $t\geq 1$, $\mathcal{D}_{x,t}G_j(x,t; y,0)$ enjoys the same bound as 
stated for $G_R(x,t;y,0)$, whereas $\mathcal{D}_{x,t}e(x,t)$ is of Gaussian order. The estimate \eqref{inconv2} thus follows similarly as
in the previous case.
\end{proof}

\begin{proof}[Proof of Proposition~\ref{prop-iteration}]
Using the above convolution estimates, the nonlinear iteration follows trivially. Indeed, by using \eqref{conv1}, \eqref{inconv1}, and the nonlinear estimate \eqref{non-est}, equation \eqref{def-TTT} then implies  
$$ \begin{aligned}
	|\mathcal{T}_v b(x,t)|  &\le  \int |G_R(x,t; y,t)  v_0 (y)|  \; \rmd y +  \int_0^t \int |G_R(x,t;y,s) Q(v,\varphi,\psi) | \; \rmd y\rmd s 
\\
&\lesssim \varepsilon \theta(x,t) +  \|b\|_1^2  \int_0^t \int |G_R(x,t;y,s) | \theta^2(y,s)| \; \rmd y\rmd s 
\\
&\lesssim \Big (\varepsilon +  \|b\|_1^2 \Big)\theta(x,t).
\end{aligned}$$
Similarly, the convolution estimates in Lemmas \ref{lem-conv} and \ref{lem-inconv} yield
 $$ |\mathcal{T}_{\varphi_{x,t}} b | +  |\mathcal{T}_{\psi_{x,t}}b | \lesssim \varepsilon \theta(x,t) + \|b\|_1^2 \theta(x,t) \lesssim (\varepsilon +\|b\|_1^2) \theta(x,t).$$

Estimates for $\mathcal{D}_{x,t}G_R$ and $\mathcal{D}_{x,t}^2 \psi$, $\mathcal{D}_{x,t}^2\varphi$ follow similarly for all terms except the 
	linear term 
$$
\int \mathcal{D}_{x,t} G_R(x,t; y,t)  v_0 (y) \; \rmd y,
$$
which is bounded by $Ct^{-1}(1+t)\theta(x,t)$ instead.
Combining, we obtain \eqref{keyh-ineq}(i).  
The contraction bound \eqref{keyh-ineq}(ii) follows by an essentially identical argument.
\end{proof}

\begin{remark}\label{layerrmk}
Note that the $t^{-1}$ singularity in $\mathcal{D}_{x,t}v$ at $t=0$ plays no role in the iteration, since $\mathcal{D}_{x,t}v$ contributes to $Q$
only where $\mathcal{D}_{x,t}\psi, \mathcal{D}_{x,t}\varphi\not \equiv 0$, hence only on $t\geq 1$, where $\chi \not \equiv 0$.
\end{remark}

\subsection{Recovery of the space and time shifts}\label{s:phase}
Applying the Banach fixed-point theorem, we obtain from Proposition \ref{prop-iteration} existence of a solution
$b=(v,\mathcal{D}_{x,t}\psi,\mathcal{D}_{x,t}\varphi)$ of the system of
integral equations $b=\mathcal{T}b$.
From the derivative estimates afforded by \eqref{keyh-ineq},
we now recover certain information about the time and space shifts $\psi$ and $\varphi$. Proposition \ref{prop-iteration} together with the definition of $\|\cdot\|_1$ yields 
$$\Big(|v| + |\mathcal{D}_{x,t}\varphi| + |\mathcal{D}_{x,t}\psi| + |v_x| + |D^2_{x,t} \varphi| + |D^2_{x,t}\psi| \Big) (x,t) \le C \varepsilon \theta(x,t). $$

\begin{lemma}\label{lem-estshifts} Let $\varphi,\psi$ be defined as in \eqref{def-shifts}. Then, there are smooth functions $\delta_{\psi,\varphi}(t)$ 
converging exponentially in time to constants $\delta_{\psi,\infty}, \delta_{\varphi,\infty}$ 
such that, for $t\geq 0$, $C>0$,
$$ |\psi(x,t) - \delta_\psi(t) e(x,t)| + |\varphi(x,t) - \delta_\varphi(t) e(x,t)| \le C \varepsilon (1+t)^{\frac12}\theta(x,t).$$
\end{lemma}
\begin{proof} We prove the estimate for $\psi$; the estimate for $\varphi$ is similar. Recall that 
$$\begin{aligned}
 \psi(x,t) &= -  \int  E_1(x,t;y,0) v_0 (y) \; \rmd y - \int_0^t \int  E_1(x,t;y,s)Q(v,\varphi,\psi) (y,s)  \rmd y\rmd s, \\
 \end{aligned}
 $$
with $E_1 (x,t; y,s) =  e(x,t-s) \beta_1(y) + G_1(x,t; y,s).$
Notice that the localization property of $\beta_1(y)$ implies
$$ e(x,t-s) \beta_1(y) \theta(y,s)^2 \quad \lesssim \quad  e(x,t-s) \rme^{-\eta (|y| + s)} = \Big[ e(x,t) + \theta(x-y,t-s) \Big] \rme^{-\eta (|y|+s)}.$$
See, for instance, \cite[Proof of Lemma~6.6]{BNSZ2} for similar estimates. This proves that
$$ \Big[ e(x,t-s) - e(x,t)\Big] \beta_1(y) \theta(y,s)^2 \quad\lesssim\quad \theta(x-y,t-s) \rme^{-\eta (|y|+s)} .$$
We can therefore write 
$$\begin{aligned}
 \psi(x,t)  &= - e(x,t)\Big[ \int \beta_1(y) v_0(y)\; \rmd y + \int_0^t \int  \beta_1(y)Q(v,\varphi,\psi) (y,s) \rmd y\rmd s \Big] 
 \\
 & \quad  -  \int  \widetilde G_1(x,t;y,0) v_0 (y) \; \rmd y - \int_0^t \int  \widetilde G_1(x,t;y,s)Q(v,\varphi,\psi) (y,s)  \rmd y\rmd s
 \end{aligned}$$
where $\widetilde G_1(x,t; y,s)$ enjoys the same bound :s stated for $ G_1(x,t; y,s)$. 
Letting 
$$ \delta_\psi(t): = - \int \beta_1(y) v_0(y)\; \rmd y - \int_0^t \int  \beta_1(y)Q(v,\varphi,\psi) (y,s) \rmd y\rmd s, $$ 
we find that
$\delta_\psi (t)$ converges exponentially to some constant $\delta_{\psi,\infty}$, since $\beta_1(y)$ is localized. That is, the first term in $\psi(x,t)$ converges exponentially in time to a plateau of shape $e(x,t)$. Following the proofs of Lemmas \ref{lem-conv} and \ref{lem-inconv}, 
we can show that 
$$\Big| \int  \widetilde G_1(x,t;y,0) v_0 (y) \; \rmd y + \int_0^t \int  \widetilde G_1(x,t;y,s)Q(v,\varphi,\psi) (y,s)  \rmd y\rmd s\Big| \le C \varepsilon (1+t)^{\frac12} \theta(x,t),$$
upon noting that $|Q(v,\varphi,\psi) (y,s)|\le C\varepsilon \theta^2(y,s)$ and $|v_0(y)|\le C\varepsilon \rme^{-|y|^2/C_0}$. 
\end{proof}

\begin{proof}[Proof of Theorem \ref{main}] The main theorem now follows as a direct consequence of Proposition~\ref{prop-iteration} (precisely, the inequality $\|v,\mathcal{D}_{x,t}\psi, \mathcal{D}_{x,t}\varphi\|_1\leq 2C_1\varepsilon $), and Lemma \ref{lem-estshifts}. 
\end{proof}


\section{Discussion and open problems}

We end by mentioning a few open problems related to our results.

First, one of the main new observations we made in this paper is the realization that the remainder term $G_R$ in the decomposition of the Green's function $G$ into $\bar{u}_x E_1 + \bar{u}_t E_2 + G_R$ behaves like a differentiated Gaussian, instead of just a Gaussian. As mentioned earlier, this is a critical difference from the analysis carried out for the complex Ginzburg--Landau equation in \cite{BNSZ1,BNSZ2}. Besides both simplifying and extending to the general reaction-diffusion case, this observation seems to make possible also the treatment of more general types of defects that arise as solutions of conservation laws. The physical interest here is that the relevant modulation equations are systems of viscous conservation laws rather than scalar equations, so that defects can occur that are analogous to the larger family of possible shocks occurring for systems. The mathematical interest comes from the increased technical difficulty presented by systems of conservation laws as opposed to a scalar equation. Pursuing this would be a very interesting problem for further investigation. As a guide in this direction, we mention the treatment in \cite{JNRZ3} of the case corresponding to localized perturbations of a constant solution of a modulating system of conservation laws, corresponding to far-field behavior of defect type solutions of conservation laws. We note that even the existence of such solutions is an open problem.

A further interesting problem, suggested by \cite{JNRZ1,JNRZ2,JNRZ3}, would be to incorporate next-order modulations of phase and wavenumber to obtain a more detailed description of the phase dynamics, as obtained for the complex Ginzburg--Landau case by different methods in \cite{BNSZ2}.

A different direction is the extension to reaction-diffusion systems with $2r$-parabolic diffusion with $r=1,2,3,\ldots$.  As noted above, our results extend to general parabolic systems with $r=1$. They clearly do not extend as stated to the case $r\geq2$, since spatial decay of solutions is in this case not Gaussian in the far field; however, we conjecture that they hold up to an additional spatially and temporally exponentially decaying error term $\rmO(\rme^{-\eta_1 (x+t)})$ for some $\eta_1>0$.

As a final open problem, we mention the stability of contact defects, which remain as the only generically occurring type not yet treated. As seen in \cite{SS04b}, due to the slow $1/|x|$ convergence of these solutions to their asymptotic states as $x\to \pm \infty$, contact defects present interesting technical challenges already at the level of spectral stability, requiring tools that are different from the non-characteristic case. Their nonlinear stability analysis should be similarly interesting and can likewise be expected to require new tools for its resolution.

\paragraph{Acknowledgments}
The authors gratefully acknowledge partial support by the NSF through the grants DMS-1411460 (Beck), DMS-1405728 (Nguyen), DMS-0907904, DMS-1408742, \& DMS-1714429 (Sandstede), and DMS-0300487 \& DMS-0801745 (Zumbrun).



\bibliographystyle{abbrv}
\bibliography{sources-stability}

\end{document}